\documentclass[a4paper,10pt,DIV=10]{scrartcl}
\usepackage[utf8]{inputenc}

\usepackage{bbold}

\usepackage{stmaryrd}
\usepackage{physics}
\usepackage{tikz}
\usetikzlibrary{backgrounds,calc,patterns,angles,quotes}
\usepackage{siunitx}
\usepackage{subcaption}
\usepackage{hyperref}
\usepackage{verbatim}
\usepackage{cite}
\usepackage{placeins}
\usepackage{xcolor,wrapfig}
\usepackage{calligra}
\usepackage{mathtools}
\usepackage{multirow}
\usepackage{wrapfig}
\usepackage{booktabs}
\usepackage{amssymb}
\usepackage{lipsum}
\usepackage{enumitem}
\usepackage{amsthm}

\usepackage[obeyFinal,
textsize=scriptsize]{todonotes}

\usepackage[
  authormarkup=none,
  todonotes]{changes}

\newcommand{\stkout}[1]{\ifmmode\text{\sout{\ensuremath{#1}}}\else\sout{#1}\fi}

\newcounter{ADD}
\newcounter{DEL}
\def\ADDin{\addtocounter{ADD}{1}}
\def\ADDout{\addtocounter{ADD}{-1}}
\def\DELin{\addtocounter{DEL}{1}}
\def\DELout{\addtocounter{DEL}{-1}}

\makeatletter
\newcommand{\ifdraft}[2]{}
\ifthenelse{\boolean{Changes@optiondraft}}{%
  \renewcommand{\ifdraft}[2]{#1}}{%
  \renewcommand{\ifdraft}[2]{#2}}
\newcommand{\ifADD}[2]{\ifnum\value{ADD}>0{#1}\else{#2}\fi}
\newcommand{\ifDEL}[2]{\ifnum\value{DEL}>0{#1}\else{#2}\fi}
\makeatother

\newcommand{\delmarkup}[1]{%
  \DELin%
  \textcolor{red!80!black}{\stkout{#1}}%
  \DELout%
}
\newcommand{\addmarkup}[1]{%
  \ADDin%
  \ifDEL{}{%
    \textcolor{green!50!black}{#1}}
  \ADDout%
}
\ifdraft{%
  \newenvironment{addedenv}{%
    \ADDin%
    \color{green!50!black}}
  {%
    \ADDout%
    \color{black}}
}
{
  
}
\ifdraft{
  \setdeletedmarkup{\delmarkup{#1}}
  \setaddedmarkup{\addmarkup{#1}}
}{}

\ifdraft{
  \newcommand{\deleteX}[1]{%
    \DELin%
    {\color{red!80!black}{#1}}%
    \DELout}%
}{
  \newcommand{\deleteX}[1]{}
}

\ifdraft{}{%
  
}

\newcommand{\jump}[1]{\left\llbracket #1 \right\rrbracket}

\newcommand{\opext}{\mathcal{L}^{\operatorname{ext}}}

\newcommand{\scpL}[2]{\left( #1, #2 \right)_{L^2}}
\newcommand{\scpLscalar}[2]{\left( #1, #2 \right)_{L^2(\Omega)}}
\newcommand{\abso}[1]{\left\vert #1 \right\vert}

\newlist{todolist}{itemize}{2}
\setlist[todolist]{label=$\square$}

\newcommand{\CFLfac}{\nu}

\usepackage{cancel}
\newcommand{\Th}{%
  \mathcal{{M}}_h%
}

\newcommand{\bfu}{\mathbf{u}}
\newcommand{\bfv}{\mathbf{v}}
\newcommand{\bfw}{\mathbf{w}}
\newcommand{\bff}{\mathbf{f}}
\newcommand{\bfA}{\mathbf{A}}

\newcommand{\bfq}{\mathbf{Q}}
\newcommand{\matrixbfK}{\mathbf{K}}
\newcommand{\matrixbfL}{\mathbf{L}}
\newcommand{\matrixbfR}{\mathbf{R}}
\newcommand{\matrixK}{K}
\newcommand{\matrixL}{L}
\newcommand{\matrixR}{R}
\newcommand{\Tone}{\mathbb{T}_1}
\newcommand{\Ttwo}{\mathbb{T}_2}
\newcommand{\Tthree}{\mathbb{T}_3}
\newcommand{\Tfour}{\mathbb{T}_4}
\newcommand{\Tfive}{\mathbb{T}_5}

\newcommand{\bflambda}{\mathbf{\Lambda}}
\newcommand{\bfIplus}{\mathbf{I}^+}
\newcommand{\bfIminus}{\mathbf{I}^-}
\newcommand{\kk}{k}
\newcommand{\Kone}{k_{1}}
\newcommand{\Ktwo}{k_{2}}
\newcommand{\Kcut}{\text{cut}}
\newcommand{\jj}{j}
\newcommand{\ubar}{\bar{u}}
\newcommand{\xjplus}{x_{\jj+\frac{1}{2}}^+}
\newcommand{\xjminus}{x_{\jj+\frac{1}{2}}^-}

\newcommand{\etaKone}{\eta_{\Kone}}

\newcommand{\numflux}[2]{\mathcal{H}(#1,#2)}
\newcommand{\numfluxwo}{\mathcal{H}}
\newcommand{\numfluxa}[2]{\mathcal{H}_a(#1,#2)}
\newcommand{\numfluxb}[2]{\mathcal{H}_b(#1,#2)}
\newcommand{\smax}{\lambda_{\max}}
\newcommand{\Source}{\mathcal{S}}
\newcommand{\Vhp}{\mathcal{V}_h^p}
\newcommand{\Iall}{\mathcal{I}_{\text{all}}}
\newcommand{\Iequi}{\mathcal{I}_{\text{equi}}}

\newcommand{\Ineigh}{\mathcal{I}_{\mathcal{N}}}
\newcommand{\tildem}{\tilde{m}}

\usepackage{enumitem,amssymb}

\usepackage{pifont}

\newtheorem{theorem}{Theorem}

\newtheorem{remark}{Remark}[section]
\newtheorem{notation}{Notation}[section]
\newtheorem{prerequisite}{Prerequisite}[section]
\newtheorem{definition}[theorem]{Definition}

\definecolor{mid2-gray}{gray}{0.55}

\title{DoD Stabilization for non-linear hyperbolic conservation laws on cut cell meshes in one dimension}
\author{Sandra May\thanks{Department of Mathematics, TU Dortmund University, Germany} \and Florian Streitb\"urger\footnotemark[1]}
\date{}

\begin{document}

\maketitle



  

%
%

%
%


\begin{abstract}
    In this work, we present the Domain of Dependence (DoD) stabilization for systems of hyperbolic conservation laws in one space dimension. The base scheme uses a method of lines approach consisting of a discontinuous Galerkin  scheme in space and an explicit strong stability preserving Runge-Kutta scheme in time. When applied on a cut cell mesh with a time step length that is appropriate for the size of the larger background cells, one encounters stability issues. The DoD stabilization consists of penalty terms that are designed to address these problems by redistributing mass between the inflow and outflow neighbors of small cut cells in a physical way. 
    For piecewise constant polynomials in space and explicit Euler in time, the stabilized scheme is monotone for scalar problems. For higher polynomial degrees $p$, our numerical experiments show convergence orders of $p+1$ for smooth flow and robust behavior in the presence of shocks.
\end{abstract}


\section{Introduction}

The efficient and fast generation of body-fitted meshes for complex geometries remains one of the most time-consuming preprocessing steps in numerical simulations involving finite volume (FV) and discontinuous Galerkin (DG) schemes. As a result, the usage of Cartesian embedded boundary meshes becomes more and more popular. 
Out of the different existing variants, we use the following approach: We simply cut the geometry out of an underlying Cartesian mesh, resulting in so called \textit{cut cells} along the boundary of the object. 

Cut cells are typically irregular and can become arbitrarily small. This causes various problems. In the context of solving hyperbolic conservation laws on cut cell meshes, for which one typically uses \textit{explicit} time stepping schemes, the most severe problem is the
so called \textit{small cell problem}: choosing the time step based on the size of the larger background cells results in stability problems on small cut cells and their neighbors. 
Therefore, special methods must be developed. The focus of this contribution is on addressing this problem.
For more information on the small cell problem we refer to \cite{BERGER2017,FVCA_May}.

The supposedly easiest approach to overcoming the small cell problem is \textit{cell merging} or \textit{cell agglomeration}
\cite{Krivodonova2013, Kummer2016,Quirk1994}: one simply merges cut cells that are too small with bigger neighbors. This approach is very intuitive but very difficult to do in three dimensions in a robust way and puts all the complexity back into the mesh generation process. 

The alternative is to develop algorithmic solutions to the small cell problem.
In the context of FV schemes, two well established approaches are the \textit{flux redistribution} method
\cite{Chern_Colella,Colella2006} and the \textit{h-box} method
\cite{Berger_Helzel_2012,Berger_Helzel_Leveque_2002}. 
More recent approaches include a \textit{dimensionally split} flux stabilization \cite{Klein_cutcell_3d,Klein_cutcell},
the \textit{mixed explicit implicit} scheme
\cite{May_Berger_explimpl}, the extension of the \textit{active flux} method to cut cells \cite{FVCA_Helzel_Kerkmann}, and the \textit{state redistribution} method \cite{Berger_Giuliani_2021}.

In the context of DG schemes there exists only very little work addressing the small cell problem.
While there are many different approaches for stabilizing discretizations for elliptic and parabolic problems on cut cell meshes (for an overview see, e.g., \cite{OverviewUCLWorkshop}) the research for hyperbolic problems
is still at the beginning but with a lot of current activity. 
Some very recent work \cite{Massing2018,Kreiss_Fu,Sticko_Kreiss} is based on applying the ghost penalty stabilization \cite{Burman2010}, which is a well-known approach for elliptic equations, to hyperbolic problems. 
Out of these contributions, only Fu and Kreiss \cite{Kreiss_Fu} address the small cell problem for first-order hyperbolic problems
by developing a stabilization for the solution of scalar conservation laws in one dimension.
A different approach to overcoming the small cell problem was taken by Giuliani \cite{Giuliani_DG} who extends the state redistribution scheme
 to the DG setting. This approach seems to work well in practice but it is challenging to verify theoretical properties.

 In \cite{DoD_SIAM_2020}, we introduced together with Engwer and N\"u{\ss}ing the \textit{Domain of Dependence} (DoD) stabilization. To the best of our knowledge, this is the first contribution to overcoming the small cell problem in a DG setting in a monotone way.
 The DoD stabilization introduces penalty terms that shift mass between small cut cells and their neighbors in a physical way: within one time step, mass is transported from the \textit{inflow} neighbors of small cut cells \textit{through} the cut cells to their \textit{outflow} neighbors. This way we restore the proper domain of dependence of the outflow neighbors and create a stable update on small cut cells for standard explicit time stepping.

 The work in \cite{DoD_SIAM_2020} treats the case of linear advection for piecewise linear polynomials. In this contribution we take the next step and extend the stabilization to higher order polynomials and to non-linear systems of hyperbolic conservation laws in one dimension, in particular to the compressible Euler equations. 
 For the extension to higher order polynomials we observed that it is not sufficient to penalize derivatives only on small cut cells. We therefore added terms to control derivatives on their neighbors as well.
 For the extension to non-linear systems, the main challenge consisted in accounting for the various flow directions.
 
 For scalar conservation laws, our extended formulation has the following theoretical properties:
 for piecewise constant polynomials in space combined with explicit Euler in time, the resulting scheme is monotone, independent of the size of the small cut cell; thus, this result transfers from the linear to the non-linear case. For the semi-discrete setting, there holds $L^2$ stability for arbitrary polynomial degrees $p$ as a result of also controlling derivatives on cut cells' neighbors. Our numerical results for scalar equations and systems show 
 convergence rates of
$p+1$ for polynomials of degree $p$ for smooth solutions and robust behavior for problems involving shocks.

 The paper is structured as follows: we will first provide in section \ref{sec: setting} the general setting, which includes the cut cell model problem and the unstabilized DG discretization. In section \ref{sec: DoD stab}, we will present the DoD stabilization for non-linear problems and higher order polynomials. We will also give a short comparison between the new formulation
 and the formulation in \cite{DoD_SIAM_2020} for the case of the advection equation.  Section \ref{sec: theoretical results} contains theoretical results for scalar conservation laws, like the monotonicity property and the $L^2$ stability result for the semi-discrete formulation.
 Finally, in section \ref{sec: numerical results} we will present numerical results for scalar equations and systems of conservation laws to support our theoretical findings. We will conclude with an outlook in section \ref{sec: outlook}.

\section{Setting}\label{sec: setting}
We consider time-dependent systems of hyperbolic conservation laws in one space dimension of the form 
\begin{equation}\label{eq: conservation law}
\bfu_t+ \bff(\bfu)_x = \mathbf{0}\quad \text{ in } \Omega \times (0,T)
\end{equation}
with initial data $\bfu_0=\bfu(\cdot,0)$.
The spatial domain is given by $\Omega = (x_L,x_R)$ with $x_L,x_R \in \mathbb{R}, x_L < x_R$, and the final time is given by $T\in \mathbb{R_+}$. Further, $\bfu	: \Omega\times(0,T) \rightarrow \mathbb{R}^m$, $m \in \mathbb{N}$, is the vector of conserved variables and $\bff:\mathbb{R}^m\rightarrow\mathbb{R}^m$ is the flux function. We assume the system to be hyperbolic, i.e., that the Jacobian $\bff_{\bfu}(\bfu)$ is diagonalizable with real eigenvalues for each physically relevant value $\bfu \in \mathbb{R}^m$, compare LeVeque \cite{Leveque02}.

In particular, we will consider the compressible Euler equations,
which satisfy \eqref{eq: conservation law} with
\begin{equation}\label{eq: Euler equations}
    \bfu = \begin{pmatrix}
    \rho \\ \rho v \\ E
    \end{pmatrix}
    \quad \text{and} \quad
    \bff(\bfu) = \begin{pmatrix} \rho v \\ \rho v^2 + p \\ (E + p)v  \end{pmatrix}.
\end{equation}
Here, $\rho$ denotes the density, $v$ the velocity, $p$ the pressure, and $E$ the energy. 
The system is completed by the equation of state 
\begin{equation*}
E = \frac{p}{\gamma-1} + \frac 1 2 \rho v^2.
\end{equation*}
We will set $\gamma=1.4$ in our numerical tests. 
We will also consider linear systems given by
\begin{equation}\label{eq: lin system}
\bfu + \bfA \bfu_x = \mathbf{0},    
\end{equation}
with the matrix $\bfA \in \mathbb{R}^m \times \mathbb{R}^m$ being 
diagonalizable with real eigenvalues.

For the theoretical results, we will focus on scalar conservation laws 
\begin{equation}\label{eq: scalar cons law}
u_t + f(u)_x = 0.
\end{equation}
Important representatives include
the linear advection equation given by
\begin{equation}\label{eq: lin adv}
    u_t + \beta u_x = 0, \quad \beta > 0 \text{ constant,}
\end{equation}
and 
Burgers equation given by
\begin{equation*}
u_t + f(u)_x = 0, \quad f(u) = \frac 1 2 u^2.    
\end{equation*}

\subsection{The cut cell model problem}\label{sec: model problem}

To examine the behavior of solving \eqref{eq: conservation law} on a cut cell mesh, we create a model problem:
We first discretize $\Omega$ in $N$ cells $I_j = (x_{j-\frac{1}{2}},x_{j+\frac{1}{2}}), j=1\ldots,N,$ of equal length $h = \frac{x_R - x_L}{N}$. 
Then we take one cell, the cell $I_k$, in the interior of the domain and split it into two cut cells, $I_{\Kone}$ and $I_{\Ktwo}$, of lengths $\alpha h$ and $(1-\alpha)h$ with $\alpha\in (0,\frac{1}{2}]$. 
This way we obtain a one dimensional cut cell mesh shown in figure \ref{fig: model problem} with $N+1$ cells, which we will refer to as $\Th$; compare also \cite{DoD_SIAM_2020}. 

 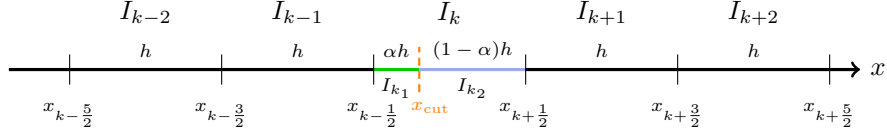
\begin{figure}[ht]
  \begin{center}
\begin{tikzpicture}[
axis/.style={very thick, line join=miter, ->}]
\draw [very thick] (-4.8,0) -- (0,0);
\draw [green!80!black,very thick] (0,0) -- (0.6,0);
\draw [blue!80!green!40!white,very thick] (0.6,0) -- (2,0);
\draw [axis] (2,0) -- (6.4,0) node(xline)[right] {$x$};
\draw (-4,-0.2) -- (-4,0.2);
\draw (-2,-0.2) -- (-2,0.2);
\draw ( 0,-0.2) -- ( 0,0.2);
\draw ( 2,-0.2) -- ( 2,0.2);
\draw ( 4,-0.2) -- ( 4,0.2);
\draw ( 6,-0.2) -- ( 6,0.2);
\draw[color=orange,densely dashed, thick] (0.6,-0.3) -- (0.6,0.3);
\node[] at (-3,0.75) {$I_{k-2}$};
\node[] at (-1,0.75) {$I_{k-1}$};
\node[] at (1 ,0.75) {$I_{k}$};
\node[] at (3 ,0.75) {$I_{k+1}$};
\node[] at (5 ,0.75) {$I_{k+2}$};
\scriptsize
\node[] at (.3 ,-0.25) {$I_{k_1}$};
\node[] at (1.3 ,-0.25) {$I_{k_2}$};
\node[] at ( -3,0.25) {$h$};
\node[] at ( -1,0.25) {$h$};
\node[] at (0.3,0.25) {$\alpha h$};
\node[] at (1.3,0.25) {$(1-\alpha) h$};
\node[] at (  3,0.25) {$h$};
\node[] at (  5,0.25) {$h$};
\node[] at (-4,-0.6) {$x_{k-\tfrac{5}{2}}$};
\node[] at (-2,-0.6) {$x_{k-\tfrac{3}{2}}$};
\node[] at ( 0,-0.6) {$x_{k-\tfrac{1}{2}}$};
\node[] at ( 2,-0.6) {$x_{k+\tfrac{1}{2}}$};
\node[] at ( 4,-0.6) {$x_{k+\tfrac{3}{2}}$};
\node[] at ( 6,-0.6) {$x_{k+\tfrac{5}{2}}$};
\node[color=orange] at (.75,-0.5) {$x_{\Kcut}$};
\end{tikzpicture}
  \caption{Cut cell mesh $\Th$: equidistant mesh with cell $I_{\kk}$
    split into two cells of lengths $\alpha h$ and $(1-\alpha) h$ with
    $\alpha \in (0,\frac{1}{2}]$. We denote the new edge coordinate by $x_{\text{cut}}$.}
  \label{fig: model problem}
  \end{center}
\end{figure}
\begin{definition}
For the model problem $\mathcal{M}_h$, we define the following index sets
\begin{equation}
    \Iequi = \{ 1 \le j \le N \lvert j \neq k \}, \:
     \Iall = \Iequi \cup \{ \Kone, \Ktwo \}, \:
     \Ineigh = \{k-1,\Kone,\Ktwo \}.
\end{equation}
\end{definition}
Here, $\Iequi$ contains the indices of all cells of length $h$, and $\Ineigh$ contains the indices of the small cut cell $I_{\Kone}$ and its left and right neighbor.

We will use this model problem for explaining our stabilization and for the theoretical results in section \ref{sec: theoretical results}. For the numerical results in section \ref{sec: numerical results}, we will build different test cases upon this model problem, which use many cut cell pairs.

\subsection{Unstabilized RKDG scheme}

We use a Runge-Kutta DG (RKDG) approach. We first discretize
in space using a DG approach. Then we discretize in time using an \textit{explicit} 
strong stability preserving (SSP) RK scheme \cite{GottliebShu, Kraaijevanger1991}.

\begin{definition}[Discrete Function Space]
We define the discrete space $\Vhp \subset (L^2(\Omega))^m$ by
\begin{equation*}
   \Vhp = \left\{ \bfv^h \in (L^2(\Omega))^m \ \vline \  \bfv^h_l{\vert_{I_j}} \in P^p(I_j)  \right.
   \left. \text{for each component }l=1,\ldots,m \text{ and for all}\ j\in \Iall \right\},
\end{equation*}
where $P^p$ denotes the polynomial space of degree $p$. 
\end{definition}

As functions $\bfv^h\in \Vhp$ are not well-defined on cell edges, we define jumps. 

\begin{definition}[Jump]
Using the notation $x_{j+\frac{1}{2}}^{\pm} = \lim_{\varepsilon \to 0} x_{j+\frac{1}{2}} \pm \varepsilon$ we define the jump at an interior edge $x_{j+\frac{1}{2}}, 1\le j \le N-1,$ as 
\begin{equation*}
\jump{\bfv^h}_{\jj+\frac{1}{2}} = \bfv^h(\xjminus) - \bfv^h(\xjplus).
\end{equation*}
Analogously, we define $\jump{\bfv^h}_{\text{cut}} = \bfv^h(x_{\text{cut}}^-) - \bfv^h(x_{\text{cut}}^+)$.
At the boundary edges $x_{\frac 1 2}$ and $x_{N+\frac 1 2}$ we define
\begin{equation*}
    \jump{\bfv^h}_{\frac 1 2} = -\bfv^h(x_{\frac 1 2}^+) \quad \text{and} \quad \jump{\bfv^h}_{N+\frac 1 2} = \bfv^h(x_{N+\frac 1 2}^-).
\end{equation*}
\end{definition}

Our stabilization is based on extending the influence of the polynomial solutions
on cells $I_{k-1}$ and $I_{\Ktwo}$ into the small cut cell $I_{\Kone}$.
We therefore introduce an extension operator,
compare \cite{DoD_SIAM_2020}.

\begin{definition}[extension operator]\label{def: extr op}
The extension operator $\opext_j$ extends the function $\bfu^h \in \Vhp$ from a cell $I_j, j\in \Iall,$ to the whole domain $\Omega$:
\begin{equation*}
    \opext_{j}: \Vhp|_{I_j} \rightarrow P^p(\Omega) \quad
  \text{s.t. } \opext_{j}(\bfu^h) \in P^p(\Omega)
  \text{ and } \opext_{j}(\bfu^h)|_{I_j} = \bfu^h|_{I_j}.
\end{equation*}
This extension is simply given by evaluating the polynomials $\bfu_l^h|_{I_j} \in P^p(I_j), l=1\ldots,m,$ outside of their original support.
\end{definition}

\begin{notation}
In the following, we will often use the shortcut notation
\begin{equation*}
    \bfu_j(x) = \opext_{j}(\bfu^h)(x), \quad x \in \Omega,
\end{equation*}
which corresponds to evaluating the discrete polynomial function from cell $j$ at a point $x$, possibly outside of $I_j$.
If necessary, we will use the subindex $l$ to denote the $l^{\text{th}}$ component. Therefore, $\bfu_{j,l}(x)$ corresponds to
the $l^{\text{th}}$ component of $\opext_{j}(\bfu^h)(x)$.
Using this notation, one can equivalently express the jump as
\begin{equation*}
\jump{\bfv^h}_{\jj+\frac{1}{2}} = \bfv_j(x_{j+\frac{1}{2}}) - \bfv_{j+1}(x_{j+\frac{1}{2}}).
\end{equation*}
\end{notation}

We now introduce the standard, unstabilized DG scheme for system \eqref{eq: conservation law}. 
The DoD stabilization, which 
will make it possible to use explicit time stepping despite the presence of the small cut cell $I_{\Kone}$,
will be introduced in section \ref{sec: DoD stab}.
The semi-discrete problem for the mesh $\Th$ is given by: Find $\bfu^h \in  \Vhp$ such that
\begin{equation}\label{eq: scheme 1d wo stab}
  \scpL{d_t\bfu^h(t)}{\bfw^h}+a_h\left(\bfu^h(t), \bfw^h\right)= 0\quad\forall \, \bfw^h\in \Vhp,
\end{equation}
with 
\begin{multline*}
    a_h (\bfu^h,\bfw^h) = -\sum_{\jj \in \Iall} \int_{\jj} \bff(\bfu^h)\cdot \partial_x \bfw^h\dd{x}\\
    + 	\sum_{\jj=0}^{N} \numflux{\bfu_j}{\bfu_{j+1}}(x_{\jj+\frac{1}{2}})\cdot\jump{\bfw^h}_{\jj+\frac{1}{2}} +  \numflux{\bfu_{\Kone}}{\bfu_{\Ktwo}}(x_{\text{cut}})\cdot\jump{\bfw^h}_{\Kcut}.
 \end{multline*}

  Here, $\mathbf{a}\cdot\mathbf{b}$ denotes the standard scalar product in $\mathbb{R}^m$ given by $\mathbf{a}\cdot\mathbf{b} = \sum_{l=1}^m \mathbf{a}_l \mathbf{b}_l$ and $(\cdot,\cdot)_{L^2}$ denotes the standard scalar product in $(L^2(\Omega))^m$. Further, $\numflux{\mathbf{a}}{\mathbf{b}}(x)$ denotes the numerical flux function with arguments $\mathbf{a}(x)$ and $\mathbf{b}(x)$.
Finally, we incorporate boundary conditions by suitably defining $\bfu_0(x_{\frac 1 2})$ and $\bfu_{N+1}(x_{N+\frac 1 2})$ in   
$\numflux{\bfu_0}{\bfu_{1}}(x_{\frac{1}{2}})$ and in $\numflux{\bfu_{N}}{\bfu_{N+1}}(x_{N+\frac{1}{2}})$, respectively.
The choices of $\bfu_0(x_{\frac 1 2})$ and $\bfu_{N+1}(x_{N+\frac 1 2})$ will be discussed in section \ref{sec: numerical results}.

\begin{notation}
In formulae we typically refer to the $j^{\text{th}}$ cell $I_j$
by using only the letter $`j'$ for brevity, i.e., $\int_j$ corresponds to $\int_{I_j}.$
\end{notation}

\section{DoD stabilization}\label{sec: DoD stab}
To handle the small cell problem, we suggest an algebraic approach, which adds special stabilization terms, summarized in $J_h$, to the semi-discrete formulation \eqref{eq: scheme 1d wo stab}.  The resulting DoD stabilized scheme is then given by: Find $\bfu_h \in \Vhp$ such that 
   \begin{equation}\label{eq: stab. scheme}
     \scpL{d_t\bfu^h(t)}{\bfw^h}+a_h(\bfu^h(t), \bfw^h) +  J_h(\bfu^h(t),\bfw^h)
     = 0 \quad\forall\, \bfw^h\in \Vhp.
     \end{equation}
The penalty term $J_h$ is linear in the test function $\bfw^h$ and in general non-linear in the solution $\bfu^h(t)$.

\subsection{General structure of the penalty term \texorpdfstring{$J_h$}{J\_h}}
We only stabilize the smaller cut cell $I_{\Kone}$ in the model mesh $\Th$. Therefore, the stabilization is given by
\begin{equation*}
J_h(\bfu^h,\bfw^h) = J_h^{0,\Kone}(\bfu^h,\bfw^h) + J_h^{1,\Kone}(\bfu^h,\bfw^h)
\end{equation*}
with
\begin{align}\label{eq: def J0}
\begin{split}
    J_h^{0,\Kone}(\bfu^h,\bfw^h) &= \etaKone \left[\numflux{\bfu_{k-1}}{\bfu_{\Ktwo}} (x_{k-\frac{1}{2}})- \numflux{\bfu_{k-1}}{\bfu_{\Kone}}(x_{k-\frac{1}{2}})\right]\cdot\jump{\bfw^h}_{k-\frac{1}{2}} \\
&+ \etaKone \left[\numflux{\bfu_{k-1}}{\bfu_{\Ktwo}}(x_\Kcut)-\numflux{\bfu_{\Kone}}{\bfu_{\Ktwo}}(x_\Kcut)\right]\cdot\jump{\bfw^h}_\Kcut
\end{split}
\end{align}
and $J_h^{1,\Kone}$ being defined below. 
Here, $\etaKone \in \mathbb{R_+}$ is a penalty factor.
The stabilization term $J_h^{0,\Kone}$ is designed to properly redistribute mass \textit{between} the cells $I_{k-1},$ $I_{\Kone},$ and $I_{\Ktwo}.$ 
We achieve this by adding new fluxes at $x_{k-1/2}$ and $x_{\text{cut}}$, which move mass between the left neighbor $I_{k-1}$ and the small cut cell $I_{\Kone}$ and between $I_{\Kone}$ and the right neighbor $I_{\Ktwo}$, respectively. The sizes of these fluxes depend on the flux differences of a newly introduced flux $\numflux{\bfu_{k-1}}{\bfu_{\Ktwo}}(\cdot)$ and the standard fluxes $ \numflux{\bfu_{k-1}}{\bfu_{\Kone}}(\cdot)$ and $ \numflux{\bfu_{\Kone}}{\bfu_{\Ktwo}}(\cdot)$, respectively. Note that the new flux $\numflux{\bfu_{k-1}}{\bfu_{\Ktwo}}(\cdot)$ introduces a direct coupling between cells $I_{k-1}$ and $I_{\Ktwo}$.

We emphasize the symmetric structure of the two terms in $J_h^{0,\Kone}$: we add jump terms at both edges of $I_{\Kone}$, accounting for the two possible flow directions.
Note that we make use of the extrapolation operator $\opext$ here when we evaluate $\bfu_{k-1}$ and $\bfu_{\Ktwo}$ at $x_{\text{cut}}$ and $x_{k-\frac 1 2}$, respectively.

The stabilization term $J_h^{1,\Kone}$ controls the mass distribution primarily \textit{within} the small cut cell $I_{\Kone}$ and secondarily \textit{within} its neighbors $I_{k-1}$ and $I_{\Ktwo}$.
The stabilization accounts for how much mass has been moved into and out of the small cut cell $I_{\Kone}$ from and to its left and right neighbors by means of $a_h$ and $J_h^{0,\Kone}$.
The terms are derived from the
proof of the $L^2$ stability, compare Theorem \ref{theorem: l2 stability}. 
Analogously to the ansatz functions, we also extrapolate the test functions to be used within their direct neighbor but outside of their original support. The stabilization term $J_h^{1,\Kone}$ is given by
\begin{align}\label{eq: def J1}
\begin{split}
J_h^{1,\Kone}(\bfu^h,\bfw^h) &= 
   \etaKone\sum_{j\in\Ineigh} \matrixbfK(j)\int_{\Kone}\left(\numflux{\bfu_{k-1}}{\bfu_{\Ktwo}}-\bff(\bfu_j)\right)\cdot\partial_x \bfw_{j}\dd x\\
   &+\etaKone\sum_{j\in \Ineigh} \matrixbfK(j)\int_{\Kone}\left(\numfluxa{\bfu_{k-1}}{\bfu_{\Ktwo}}\bfu_j\right)\cdot\partial_x \bfw_{k-1}\dd x\\
   &+\etaKone\sum_{j\in \Ineigh} \matrixbfK(j)\int_{\Kone}\left(\numfluxb{\bfu_{k-1}}{\bfu_{\Ktwo}}\bfu_j\right)\cdot\partial_x \bfw_{\Ktwo}\dd x.
   \end{split}
\end{align}
Here, the matrices
$\matrixbfK(j) \in \mathbb{R}^{m \times m}, j \in \Ineigh,$ incorporate information about the flow directions. They are defined using positive semi-definite matrices
$\matrixbfL_{\Kone},\matrixbfR_{\Kone} \in \mathbb{R}^{m \times m}$ and the identity matrix 
$\mathbf{I}^m \in \mathbb{R}^{m \times m}$.
We set
\begin{equation*}
\matrixbfK(k-1)=\matrixbfL_{\Kone}, \quad \matrixbfK(\Kone)=-\mathbf{I}^m, \quad \text{and} \quad \matrixbfK(\Ktwo)=\matrixbfR_{\Kone}.
\end{equation*}
The choices of $\matrixbfL_{\Kone},\matrixbfR_{\Kone}, \text{and } \etaKone$ will be discussed below.

Further, $\numfluxa{\bfu^-}{\bfu^+} \in \mathbb{R}^{m \times m}$ denotes the Jacobian of the numerical flux $\numflux{\bfu^-}{\bfu^+}$ with respect to the first argument, i.e., $\left(\frac{\partial}{\partial (\bfu^-)_j} \numflux{\bfu^-}{\bfu^+}_i\right)_{i,j=1}^m$. Analogously, $\numfluxb{\bfu^-}{\bfu^+}$ denotes the Jacobian with respect to the second argument $\bfu^+$.

The stabilization might seem a bit overwhelming. Below, we will examine the stabilization for the two special cases of linear advection for $P^p$ and of scalar conservation laws for $P^0$ in more detail. This will provide a better understanding. 

\begin{remark}
We note that the stabilized DG scheme is {\em locally mass conservative} but that the local mass conservation must be understood in a slightly broader sense: when checking for mass conservation (by testing with indicator functions), the penalty term $J_h^{1,\Kone}$ vanishes. The penalty term $J_{h}^{0,\Kone}$ stays and (depending on the flow direction) connects the cells $I_{k-1}$, $I_{\Kone}$, and $I_{\Ktwo}$. This is intended to overcome the small cell problem. As a result, we have local mass conservation with respect to the extended control volume $I_{k-1} \cup I_{\Kone} \cup I_{\Ktwo}$. 
\end{remark}

\subsection{Choice of parameters}

We now discuss how to choose $\matrixbfL_{\Kone}$,  $\matrixbfR_{\Kone}$, and $\etaKone$.

\subsubsection{Choice of \texorpdfstring{$\matrixbfL_{\Kone}$}{L\_k} and \texorpdfstring{$\matrixbfR_{\Kone}$}{R\_k}}
The parameter matrices $\matrixbfL_{\Kone}$ and $\matrixbfR_{\Kone}$ incorporate information about the flow direction. 
Let us first consider \textit{linear} problems. 
For the scalar linear advection equation \eqref{eq: lin adv} with $\beta>0$, we set $\matrixL_{\Kone}=1$ and $\matrixR_{\Kone}=0.$
For linear systems, given by \eqref{eq: lin system}, we decompose the matrix $\bfA$. 
Thanks to the assumption of hyperbolicity, $\bfA$ is diagonalizable with real eigenvalues
$\lambda_{i}, i=1,\ldots,m,$. Therefore, we can rewrite 
$
\bfA = \bfq\bflambda \bfq^{-1},
$
with the columns of $\bfq$ containing the right eigenvectors of $\bfA$ and $\bflambda$ being a diagonal matrix containing the eigenvalues $(\lambda_i)_i$. 
Based on $\bflambda$, we define the diagonal matrices $\bfIplus, \bfIminus \in \mathbb{R}^{m \times m}$
by choosing element-wise for $i=1,\ldots,m$
\begin{equation*}
    \bfIplus_{ii} = \begin{cases}
      1 & \text{if } \bflambda_{ii} > 0,\\
      \frac{1}{2} & \text{if } \bflambda_{ii} = 0,\\
      0 & \text{if } \bflambda_{ii} < 0,
    \end{cases}
    \quad \text{and} \quad
    \bfIminus_{ii} = \begin{cases}
      0 & \text{if } \bflambda_{ii} > 0,\\
      \frac{1}{2} & \text{if } \bflambda_{ii} = 0,\\
      1 & \text{if } \bflambda_{ii} < 0.
    \end{cases}    
\end{equation*}
Then, we define
\begin{equation}\label{eq: choice tau kappa system}
    \matrixbfL_{\Kone} = \bfq\bfIplus \bfq^{-1} \quad \text{and} \quad \matrixbfR_{\Kone} = \bfq\bfIminus \bfq^{-1}.
\end{equation}
Note that $\matrixbfL_{\Kone}$ and $\matrixbfR_{\Kone}$ are positive semi-definite matrices, which satisfy
$\matrixbfL_{\Kone} + \matrixbfR_{\Kone} = \mathbf{I}^m$.

For \textit{non-linear} problems, we use the same approach but replace $\bfA$ by the (non-linear) Jacobian matrix $\bff_{\bfu}(\bfu)$,
evaluated at a suitable average $\hat{\bfu}$ of $\bfu_{k-1}(x_{\Kone})$ and $\bfu_{\Ktwo}(x_{\Kone})$, with $x_{\Kone}$ denoting the cell centroid of cell $I_{\Kone}$.
For scalar problems, i.e., $m=1$, we simply use the arithmetic average
$\hat{u} = (u_{k-1}(x_{\Kone})+u_{\Ktwo}(x_{\Kone}))/2$ and set
\begin{equation*}
(\matrixL_{\Kone},\matrixR_{\Kone}) = 
    \begin{cases}
     (1,0) & \text{if } \hat{u} > 0, \\
     (\frac{1}{2},\frac{1}{2}) & \text{if } \hat{u} = 0, \\
     (0,1) & \text{if } \hat{u} < 0.
    \end{cases}
\end{equation*}
For solving the compressible Euler equations, compare \eqref{eq: Euler equations}, we use the Roe average given by
\begin{equation*}
\hat{\bfu}(\bfu_{k-1},\bfu_{\Ktwo}) =\frac{1}{2} \begin{pmatrix}
\sqrt{\rho_{k-1}}+\sqrt{\rho_{\Ktwo}}\\
\sqrt{\rho_{k-1}}v_{k-1}+\sqrt{\rho_{\Ktwo}}v_{\Ktwo}\\
\sqrt{\rho_{k-1}}H_{k-1}+\sqrt{\rho_{\Ktwo}}H_{\Ktwo}
\end{pmatrix}
\end{equation*}
with $H=\frac{E+p}{\rho}$ and with dropping the evaluation point $x_{\Kone}$ for brevity. Then, we decompose $\bff_{\bfu}(\hat{\bfu})$ into $\bfq\bflambda \bfq^{-1}$ and use again the definition \eqref{eq: choice tau kappa system}.

\subsubsection{Choice of \texorpdfstring{$\etaKone$}{eta\_k1}}
We choose the stabilization parameter $\etaKone$ as
\begin{equation}\label{eq: eta}
\etaKone = \max \left( 1-\frac{\alpha}{\CFLfac}, 0 \right)
\end{equation}
with $\alpha$ being the cut cell fraction and $\CFLfac$ the CFL parameter. The CFL parameter is used for setting the time step $\Delta t$. 
We use the standard formula for computing the time step length for DG schemes given by
\begin{equation}\label{eq: time step}
\Delta t = \frac{1}{2p+1} \frac{\CFLfac h}{\smax}
\end{equation}
with $    \smax = \max_i \abs{\lambda_i} $
 being the maximum eigenvalue.

Examining \eqref{eq: eta}, we observe that for $\alpha \ge \CFLfac$ there holds $\etaKone = 0$ and therefore the stabilization $J_h$ vanishes. This is intended as in this case the standard CFL condition on cell $I_{\Kone}$ is satisfied and we do not have a small cell problem. 
In the following, we typically implicitly assume $\alpha < \nu$, in which case there holds $\etaKone = 1- \frac{\alpha}{\CFLfac} > 0$.

\begin{remark}
There is a certain (limited) flexibility in the choice of $\etaKone$. For a more detailed discussion we refer to
\cite{DoD_SIAM_2020,Enumath_proceedings}.
\end{remark}

\subsection{Effect of additional stabilization terms in \texorpdfstring{$J_h^{1,\Kone}$}{Jh1}}

We now briefly discuss our new formulation for the case of the linear advection equation and compare it to the
formulation used in \cite{DoD_SIAM_2020}, where we presented the DoD stabilization for the advection equation for piecewise \textit{linear} polynomials. We start with examining $J_h^{0,\Kone}$ as formulated in \eqref{eq: def J0}. When using an upwind flux, the first term simply cancels and the second term reduces to the formulation of $J_h^{0,\Kone}$ used in \cite{DoD_SIAM_2020}; thus, the two formulations coincide for $J_h^{0,\Kone}$.
This is not the case for $J_h^{1,\Kone}$. Compared to \cite{DoD_SIAM_2020}, we have added terms to stabilize the mass distribution within the cells $I_{k-1}, I_{\Kone}$, and $I_{\Ktwo}$ for higher order polynomials. We discuss this in more detail in the following.

For solving the linear advection equation \eqref{eq: lin adv} with the upwind flux, the derivatives of
the numerical flux are given by
\begin{equation*}
\numfluxa{u_a}{u_b} = \beta \quad \text{and} \quad \numfluxb{u_a}{u_b} = 0,
\end{equation*}
and the coefficients $\matrixL_k$ and $\matrixR_k$ reduce to $\matrixL_k = 1$ and $\matrixR_k = 0$. Thus, the stabilization is of the form 

\begin{align}\label{eq: stabilization linear advection}
\begin{split}
J_h(u^h,w^h) = & \beta \etaKone\left[u_{k-1}(x_\Kcut)-u_{\Kone}(x_\Kcut)\right]\jump{w_h}_\Kcut \\
   &+\beta \etaKone\int_{\Kone} \left[u_{k-1}(x)-u_{\Kone}(x)\right]\left[\partial_x w_{k-1}(x)-\partial_x w_{\Kone}(x)\right] \dd{x}.
   \end{split}
\end{align}
The stabilization suggested in \cite{DoD_SIAM_2020} for the same setting has the form
\begin{align}\label{eq: stab lin adv old}
\begin{split}
J_h(u^h,w^h) = & \beta \etaKone\left[u_{k-1}(x_\Kcut)-u_{\Kone}(x_\Kcut)\right]\jump{w_h}_\Kcut \\
   &\quad -\beta \etaKone\int_{\Kone} \left[u_{k-1}(x)-u_{\Kone}(x)\right]\partial_x w_{\Kone}(x) \dd{x}.
   \end{split}
\end{align}
Therefore, the only but essential difference is the expression 
\begin{equation}\label{eq: difference term siam}
\int_{\Kone}\beta\etaKone \left[u_{k-1}(x)-u_{\Kone}(x)\right]\partial_x w_{k-1}(x)\dd x.
\end{equation}

To examine the effect of the additional term \eqref{eq: difference term siam}, especially for higher polynomial degrees, we study the eigenvalues of the semi-discrete system
\begin{equation*}
U_t = M^{-1}L U. 
\end{equation*}
Here, we denote by $M$ the mass matrix, by $L$ the stabilized stiffness matrix, and by $U$ the coefficient vector of $u_h$. 
We use a modified version of our model problem $\Th$ here: we discretize the domain $(0,1)$ by 100 equidistant cells and then split all cells in $(0.1,0.9)$ in cut cell pairs of length $\alpha h$ and $(1-\alpha)h$. All cells of length $\alpha h$ are identified as cells of type $I_{\Kone}$ and are stabilized. We compare the results for $\alpha=10^{-1}$ with the results for $\alpha=10^{-6}$.

\begin{table}[t]
\centering
\begin{tabular}{|c|c|c|c|c|}
\hline
& \multicolumn{2}{|c|}{$\alpha = 10^{-1}$} & \multicolumn{2}{|c|}{$\alpha = 10^{-6}$}\\
\hline
P$^p$ &  without \eqref{eq: difference term siam} & with \eqref{eq: difference term siam} 
& without \eqref{eq: difference term siam} & with \eqref{eq: difference term siam}   \\ \hline
P$^1$ & -4.81e-16 & 4.38e-17 & -1.58e-17  &  8.40e-17 \\
P$^2$ &  2.51e-04 & -1.73e-15 & 1.10e-15 & -6.53e-16 \\
P$^3$ & 5.11e-03 & 2.50e-15  & 3.84e-17  & 2.46e-16 \\
\hline
\end{tabular}
\caption{Effect of adding the term \eqref{eq: difference term siam}: Comparison of the spectral abscissa for the modified model problem for $\alpha = 10^{-1}$ and $\alpha=10^{-6}$. }\label{Table: spectral abscissa comparison}
\end{table}

In table \ref{Table: spectral abscissa comparison} we show the spectral abscissa of $M^{-1}L$ for the two different stabilizations for polynomial degrees $p=1,2$, $3$. The spectral abscissa $\mu$ is defined as 
the supremum over the real parts of all eigenvalues $\hat{\lambda}_i$ of $M^{-1}L$, i.e., $\mu = \sup_{i}( \text{Re}(\hat{\lambda}_i))$. It is a good indicator for the stability of a semi-discrete system, compare \cite{Mitchell2020,Trefethen2005}. We need to prevent $\mu>0$.

Usually, the tiny cut cells are the trouble-makers. 
For $\alpha=10^{-6}$ though all values in table \ref{Table: spectral abscissa comparison} are zero
(within the range of machine precision) and therefore fine. 
So one might think that the formulation \eqref{eq: stab lin adv old}, which we introduced in \cite{DoD_SIAM_2020} for linear polynomials only, also works for higher order polynomials.

Surprisingly though we have problems for the `big' cut cells with
volume fraction $\alpha = 10^{-1}$. For $P^1$ the values look good for both formulations. For $P^2$ and $P^3$ however
the formulation without the term \eqref{eq: difference term siam} shows values for $\mu$ of the order of $10^{-4}$ and $10^{-3}$, i.e., a clear indication of instability. For our new formulation, which adds the term \eqref{eq: difference term siam}, the values are zero again.
When examining the term \eqref{eq: difference term siam}, we can confirm that it should be more relevant for relatively large volume fractions $\alpha$ as we integrate over cells of type $I_{\Kone}$, which have length $\alpha h$, and the derivative $\partial_x w_{k-1}$ scale like
$\mathcal{O}(1/h)$. 

This example also shows that it is important to not only focus on the case of tiny $\alpha$'s but to also ensure that everything runs stable for larger volume fractions as well.

\begin{remark}
A similar observation seems to hold true for solving the Euler equations: reducing $J_h^{1,\Kone}$ to only using the single term
\begin{equation*}
- \etaKone\int_{\Kone}\left(\numflux{\bfu_{k-1}(x)}{\bfu_{\Ktwo}(x)}-\bff(\bfu_{\Kone}(x))\right)\cdot\partial_x \bfw_{\Kone}(x)\dd x
\end{equation*}
leads to stable results for $P^1$ polynomials in our tests but causes instabilities for higher order polynomials.
\end{remark}

\subsection{Limiter}\label{sec: limiter}
To run test cases involving a shock in a stable way, we need a limiter.
We use the {\em total variation diminishing in the means} (TVDM) generalized slope limiter developed by Cockburn and Shu \cite{Cockburn1998,CockburnShu1989}, which we modify appropriately for the neighbors $I_{k-1}$ and $I_{\Ktwo}$ of the small cut cell.

The standard scheme for limiting the discrete solution $u_j$ on a cell $I_j, j \in \Iall,$ (of a non-uniform mesh)
can be summarized as follows:
\begin{enumerate}
    \item Compute the limited extrapolated values $u_j^{\lim}(x_{j-\frac{1}{2}}^+)$ and $u_j^{\lim}(x_{j+\frac{1}{2}}^-)$:
    \begin{align*}
        u_j^{\lim}(x_{j-\frac{1}{2}}^+) &= \ubar_j - \tildem(\ubar_j-u_j(x_{j-\frac{1}{2}}^+),\ubar_j-\ubar_{j-1},\ubar_{j+1}-\ubar_j)\\
        u_j^{\lim}(x_{j+\frac{1}{2}}^-) &= \ubar_j + \tildem(u_j(x_{j+\frac{1}{2}}^-)-\ubar_j,\ubar_j-\ubar_{j-1},\ubar_{j+1}-\ubar_j)
    \end{align*}
    with $\ubar_j$ denoting the average mass of $u_j$ over cell $I_j$ and $\tildem$ being the \textit{minmod} function given by
    \begin{equation*}
    \tildem(a_1,\ldots, a_n) = 
    \begin{cases}
    s \cdot \min_{1\le i \le n} \abs{a_i} & \text{if } \text{sign}\,(a_1) = \ldots = \text{sign}\,(a_n) = s,\\
    0 & \text{otherwise.}
    \end{cases}
    \end{equation*}
    \item If the limited values $u_j^{\lim}(x_{j-\frac{1}{2}}^+)$ and $u_j^{\lim}(x_{j+\frac{1}{2}}^+)$ are equal to the unlimited values $u_j(x_{j-\frac{1}{2}}^+)$ and $u_j(x_{j+\frac{1}{2}}^+)$, set $u_j^{\lim}=u_j$.
    Otherwise, reduce $u_j$ to $P^1$ by setting higher order coefficients to zero. 
    (Note that this does not change the mass as we use a Legendre basis.)
    Then, limit the linear polynomial
    such that the edge evaluations of the limited polynomial do not exceed $u_j^{\lim}(x_{j-\frac{1}{2}}^+)$ and $u_j^{\lim}(x_{j+\frac{1}{2}}^+)$, respectively. Use the outcome as $u_j^{\lim}$.
\end{enumerate}
Note that despite using the minmod function, this approach of limiting is more in the spirit of the MC limiter.

In the penalty term $J_h$, we evaluate the solutions 
of cells $I_{k-1}$ and $I_{\Ktwo}$ outside of their original support. We therefore postprocess the limiting
on  these cells to additionally enforce 
\begin{gather*}
    \min \left(\ubar_{\kk-1}^n, \ubar_{\Kone}^n, \ubar_{\Ktwo}^n \right)
    \le u_{\kk-1}(x_{\Kcut}) \le \max \left(\ubar_{\kk-1}^n, \ubar_{\Kone}^n, \ubar_{\Ktwo}^n \right),\\
    \min \left(\ubar_{\kk-1}^n, \ubar_{\Kone}^n, \ubar_{\Ktwo}^n \right)
    \le u_{\Ktwo}(x_{k-\frac{1}{2}}) \le \max \left(\ubar_{\kk-1}^n, \ubar_{\Kone}^n, \ubar_{\Ktwo}^n \right).
  \end{gather*}
As for the standard cells, we first apply a check whether it is necessary to change the high order polynomial (see Step 1) and only adjust the solution if needed.

\begin{remark}
This limiter has been adjusted to our stabilization and produces robust results but 
tends to be diffusive for higher order. 
The focus of this work is on the development of the stability term $J_h$, not on limiting.
Limiting in this setup is a very challenging task as it combines the issues of not
limiting higher order polynomials at smooth extrema and complications caused by the cut cell geometry \cite{May_Berger_LP}.
We plan to address this in future work.
\end{remark}

\section{Theoretical results}\label{sec: theoretical results}

In this section we present theoretical results concerning the stability of the stabilized scheme.
For this, we focus on \textit{scalar} conservation laws given by \eqref{eq: scalar cons law}. We also require some standard properties for the numerical flux, compare, e.g., Cockburn and Shu \cite{CockburnShu1989}.

\begin{prerequisite} \label{properties num flux}
  We request the numerical flux $\numfluxwo$ to satisfy the following properties:
  \begin{enumerate}
  	\item Consistency: $\numflux{u}{u} = f(u)$.
	\item Continuity: $\numflux{u^-}{u^+}$ is at least Lipschitz continuous with respect to both arguments $u^-$ and $u^+$. 
	\item Monotonicity: $\numflux{u^-}{u^+}$
	\begin{itemize}
		\item is a non-decreasing function of its first argument $u^-$,
		\item is a non-increasing function of its second argument $u^+$.
	\end{itemize}
  \end{enumerate}
  \end{prerequisite}
  Then, the flux has the E-flux property defined by Osher \cite{Osher1984}: For all $u$ between $u^-$ and $u^+$ there holds 
  \begin{equation}\label{eq: e flux}
      (\numflux{u^-}{u^+}-f(u))(u^+-u^-)\leq 0.
  \end{equation}

\subsection{Theoretical results for \texorpdfstring{$P^0$}{P0}}
We first consider the case of piecewise constant polynomials. Then, the stabilization $J_h^{1,\Kone}$, which involves derivatives
of the test functions, vanishes, and
the stabilization for the model problem $\Th$ reduces to $J_h(u^h,w^h)= J_h^{0,\Kone}(u^h,w^h)$. 
In time we use explicit Euler. This results in the following update formulae in the neighborhood of the small cut cell
$I_{\Kone}$ 
\begin{align}\label{eq: formula for P0}
\begin{split}
    u^{n+1}_{k-2} = & u^n_{k-2}-\frac{\Delta t}{h}\lbrace \numflux{u^n_{k-2}}{u^n_{k-1}}-\numflux{u^n_{k-3}}{u^n_{k-2}}\rbrace,\\
    u^{n+1}_{k-1} = & u^n_{k-1}-\frac{\Delta t}{h}\lbrace(1-\etaKone)\numflux{u^n_{k-1}}{u^n_{\Kone}}
    +\etaKone\numflux{u^n_{k-1}}{u^n_{\Ktwo}}-\numflux{u^n_{k-2}}{u^n_{k-1}}\rbrace,\\
    u^{n+1}_{\Kone} = & u^n_{\Kone}-\frac{\Delta t}{\alpha h}(1-\etaKone)\lbrace\numflux{u^n_{\Kone}}{u^n_{\Ktwo}}-\numflux{u^n_{k-1}}{u^n_{\Kone}}\rbrace,\\
    u^{n+1}_{\Ktwo} = & u^n_{\Ktwo}-\frac{\Delta t}{(1-\alpha)h}\lbrace\numflux{u^n_{\Ktwo}}{u^n_{k+1}}-(1-\etaKone)\numflux{u^n_{\Kone}}{u^n_{\Ktwo}}-\etaKone \numflux{u^n_{k-1}}{u^n_{\Ktwo}}\rbrace,\\
    u^{n+1}_{k+1} = & u^n_{k+1}-\frac{\Delta t}{h}\lbrace \numflux{u^n_{k+1}}{u^n_{k+2}}-\numflux{u^n_{\Ktwo}}{u^n_{k+1}}\rbrace.
\end{split}
\end{align}
We use the common FV notation and denote the solution in cell $I_j$ at time $t^n$ by $u_j^n$.
Evaluation points $'x'$ are not necessary as we only consider piecewise constant solutions.

The update formulae in \eqref{eq: formula for P0} gives some insight in the effect of the stabilization.
Let us first consider the 
update for the small cell $I_{\Kone}$. The factor $(1-\etaKone) = \frac{\alpha}{\nu}$ in front of the flux difference
balances the factor $\alpha$ (from the cell size $\alpha h$) in the denominator and provides the prerequisite for a stable update on $I_{\Kone}$. 
Further, there now exists an additional flux $\numflux{u_{k-1}}{u_{\Ktwo}}$ between the cells
$I_{k-1}$ and $I_{\Ktwo}$, which are not direct neighbors.
The scaled mass given by
$\etaKone \numflux{u_{k-1}}{u_{\Ktwo}}$ is directly transported between cells $I_{k-1}$ and $I_{\Ktwo}$ (depending on the flow direction), skipping the small cut cell $I_{\Kone}$.

\subsubsection{Monotonicity}

A standard first-order FV/DG scheme is monotone on a uniform mesh for scalar conservation laws. 
We can also show this property for our stabilized scheme on the model mesh $\Th$. This guarantees that overshoot cannot occur.
For explicit schemes, a monotone scheme can be defined as follows, compare Toro \cite{Toro}.
\begin{definition}\label{def: monotonicity}
A method
$ u^{n+1}_\jj = H(u^n_{\jj-i_L},u^n_{\jj-i_L+1},...,u^n_{\jj+i_R}) $
is called {\em monotone}, if $\:\forall j$ there holds for every $l$ with $-i_L\le l\le i_R$
\begin{equation}\label{Def_monotone_coeff}
\frac{\partial H}{\partial u_{j+l}}(u_{j-i_L},...,u_{j+i_R})\geq 0.
\end{equation}
\end{definition}
\begin{theorem}\label{theorem: monotonicity}
Consider the stabilized scheme \eqref{eq: stab. scheme} for $P^0$ polynomials for the model problem $\Th$ with explicit Euler in time, applied to a
scalar conservation law. Let the time step be given by $\Delta t = \frac{\CFLfac h}{\smax}$ for $0 < \alpha < \CFLfac < 1-\alpha$.
Let the numerical flux $\mathcal{H}$ satisfy prerequisite \ref{properties num flux}. 
Further, we require:
\begin{equation}\label{eq: numflux condition}
\abso{\numfluxa{u}{v}} + \abso{\numfluxb{w}{u}} \le \frac{\nu h}{\Delta t} \quad \forall u,v,w.
\end{equation} 
Then, the stabilized scheme is monotone.
\end{theorem}

\begin{remark}
Condition \eqref{eq: numflux condition} is a common condition for monotonicity on regular meshes, compare
\cite{Mishra_Abgrall}. 
\end{remark}

\begin{proof}
Away from the two cut cells, we use a standard first-order DG scheme on a uniform mesh, which is monotone under the given assumptions. 
It therefore suffices to show property \eqref{Def_monotone_coeff} for the three cells $I_j, j\in\Ineigh,$ that are affected by our stabilization. The update formulae are given by \eqref{eq: formula for P0}. Due to $0<\etaKone<1$, the
non-negativity of $\frac{\partial}{\partial u_i^n} u_j^{n+1}$ for $i \neq j$ follows directly from the monotonicity of the fluxes.
It remains to examine $\frac{\partial}{\partial u^n_j} u_j^{n+1}$ for $j \in \Ineigh$.
We start with cell $I_{k-1}$:
\begin{align*}
\frac{\partial}{\partial u^n_{k-1}} u^{n+1}_{k-1}
=& \,1-\frac{\Delta t}{h}\{(1-\etaKone)\numfluxa{u^n_{k-1}}{u^n_{\Kone}}+\etaKone\numfluxa{u^n_{k-1}}{u^n_{\Ktwo}}\\
&\qquad -(1-\etaKone)\numfluxb{u^n_{k-2}}{u^n_{k-1}}  -\etaKone \numfluxb{u^n_{k-2}}{u^n_{k-1}} \}\\
\geq & \,1-\frac{\Delta t}{h}\left\{ (1-\etaKone)\frac{\nu h}{\Delta t} + \etaKone \frac{\nu h}{\Delta t}  \right\}\geq 0.
\end{align*}
For the small cut cell $I_{\Kone}$ there holds with $\etaKone = 1 - \frac{\alpha}{\CFLfac}$
\begin{align*}
\frac{\partial}{\partial u^n_{\Kone}} u^{n+1}_{\Kone}
=& \, 1-\frac{\Delta t }{\alpha h}(1-\etaKone)\{\numfluxa{u^n_{\Kone}}{u^n_{\Ktwo}}-\numfluxb{u^n_{k-1}}{u^n_{\Kone}}\}\\
\geq & \, 1-\frac{\Delta t}{\alpha h}\frac{\alpha}{\CFLfac}\frac{\nu h}{\Delta t}\geq 0.
\end{align*}
Finally, for cell $I_{\Ktwo}$ we get
\begin{align*}
\frac{\partial}{\partial u^n_{\Ktwo}} u^{n+1}_{\Ktwo}
=& \, 1-\frac{\Delta t}{(1-\alpha) h}\{(1-\etaKone)\numfluxa{u^n_{\Ktwo}}{u^n_{k+1}} + \etaKone\numfluxa{u^n_{\Ktwo}}{u^n_{k+1}}\\
&\qquad \qquad\quad -(1-\etaKone)\numfluxb{u^n_{\Kone}}{u^n_{\Ktwo}} -\etaKone\numfluxb{u^n_{k-1}}{u^n_{\Ktwo}}\}\\
\geq& \, 1-\frac{\Delta t}{(1-\alpha) h} \frac{\nu h}{\Delta t} \geq 0.
\end{align*}
This concludes the proof.
\end{proof}

\subsection{\texorpdfstring{$L^2$}{L2} stability for \texorpdfstring{$P^p, p \ge 0$}{Pp, p ge 0}}
In this section we prove that the stabilized semi-discrete scheme \eqref{eq: stab. scheme} is $L^2$ stable for arbitrary polynomial degree $p$ for the model problem $\Th$. The time is not discretized here and we consider scalar conservation laws.

We note that the \textit{un}stabilized semi-discrete scheme \eqref{eq: scheme 1d wo stab} is also $L^2$ stable in this setting as shown in the proof below. But when combined with an explicit time stepping scheme, one would need to take tiny time steps to ensure stability for the fully discrete scheme.
This is not the case for our stabilized scheme. 
The difficulty in designing the stabilization term $J_h$ is to find a formulation that is both $L^2$ stable for the semi-discrete setting and solves the small cell problem for the fully discrete setting in a monotone way.
\begin{theorem}\label{theorem: l2 stability}
Let $u^h(t)$, with $u^h(t) \in \Vhp$ for any fixed $t$, be the solution to the semi-discrete problem \eqref{eq: stab. scheme} for the scalar equation \eqref{eq: scalar cons law} with periodic boundary conditions. Let the numerical flux function $\mathcal{H}$ satisfy prerequisite \ref{properties num flux}.
Then, the solution satisfies for all $t\in (0,T)$
\begin{equation*}
    \norm{u^h(t)}_{L^2(\Omega)} \leq \norm{u^h(0)}_{L^2(\Omega)}.
\end{equation*}
\end{theorem}

\begin{proof}
We choose $w^h = u^h(t)$ in \eqref{eq: stab. scheme} to get 
\begin{equation*}
 \scpLscalar{d_t u^h(t)}{u^h(t)}+a_h(u^h(t), u^h(t)) +  J_h(u^h(t),u^h(t)) = 0.
\end{equation*}
We integrate in time to get for the first term
\begin{equation*}
\int_{0}^t \scpLscalar{d_{\tau} u^h(\tau)}{u^h(\tau)} \: \dd \tau = 
\int_0^t \frac{d}{d\tau} \frac{1}{2} \norm{u^h(\tau)}_{L^2(\Omega)}^2 \dd \tau 
= \frac 1 2 \norm{u^h(t)}_{L^2(\Omega)}^2 - \frac 1 2  \norm{u^h(0)}_{L^2(\Omega)}^2.
\end{equation*}
It remains to show that for any fixed $t$
\begin{equation*}
a_h(u^h(t), u^h(t)) +  J_h(u^h(t),u^h(t)) \ge 0.
\end{equation*}
In the following we will suppress the explicit time dependence for brevity.

\textbf{Unstabilized case:} We first prove $L^2$ stability for the unstabilized case, i.e., we show $a_h(u^h,u^h) \ge 0$.
Here, we follow Jiang and Shu \cite{JiangShu} for the special case of the square entropy function.
We define 
\begin{equation*}
g(u) = \int^u f(\hat{u})\dd \hat{u}.
\end{equation*}
This implies $g'(u) = f(u)$. By the E-flux property \eqref{eq: e flux} and the mean value theorem, there holds 
\begin{equation}\label{eq: relation H and f}
    \numflux{u^-}{u^+} (u^- - u^+) - (g(u^-) - g(u^+)) \ge 0.
\end{equation}
Further, there holds for an arbitrary cell $I_i$ and an arbitrary $u_j$
\begin{equation*}
\int_i f(u_j) \: \partial_x u_j \dd x = g(u_j(x_{i+\frac{1}{2}}))-g(u_j(x_{i-\frac{1}{2}})).
\end{equation*}
We define the flux
\begin{equation*}
F_{i+\frac{1}{2}}(u) = \numflux{u_{i}}{u_{i+1}}(x_{i+\frac{1}{2}})\; u_{i}(x_{i+\frac{1}{2}})- g(u_i(x_{i+\frac{1}{2}})).
\end{equation*}
Then we can rewrite the contribution of the bilinear form $a^h$ for a single, arbitrary cell $I_i$ as
\begin{align*}
-\int_i &f(u_i(x)) \partial_x u_i(x)\dd x + \numflux{u_i}{u_{i+1}}(x_{i+\frac{1}{2}}) \: u_i(x_{i+\frac{1}{2}})
- \numflux{u_{i-1}}{u_{i}}(x_{i-\frac{1}{2}}) \: u_i(x_{i-\frac{1}{2}}) \\
=&- g(u_i(x_{i+\frac{1}{2}}))+g(u_i(x_{i-\frac{1}{2}})) + \numflux{u_i}{u_{i+1}}(x_{i+\frac{1}{2}}) \: u_i(x_{i+\frac{1}{2}})
- \numflux{u_{i-1}}{u_{i}}(x_{i-\frac{1}{2}}) \: u_i(x_{i-\frac{1}{2}}) \\
=& F_{i+\frac{1}{2}}(u) + g(u_i(x_{i-\frac{1}{2}}))
- \numflux{u_{i-1}}{u_{i}}(x_{i-\frac{1}{2}}) \: u_i(x_{i-\frac{1}{2}}) \\
=& F_{i+\frac{1}{2}}(u) - F_{i-\frac{1}{2}}(u) - g(u_{i-1}(x_{i-\frac{1}{2}})) + g(u_i(x_{i-\frac{1}{2}}))
+ \numflux{u_{i-1}}{u_{i}}(x_{i-\frac{1}{2}}) \jump{u^h}_{i-\frac{1}{2}}.
\end{align*}
Using the notation $\jump{g(u)}_{i+\frac{1}{2}} = g(u_{i}(x_{i+\frac{1}{2}})) - g(u_{i+1}(x_{i+\frac{1}{2}}))$, 
we can summarize
\begin{align*}
a_h (u^h,u^h) =& \sum_{\jj \in \Iequi}
\left( F_{\jj+\frac{1}{2}}(u) - F_{\jj-\frac{1}{2}}(u)
+ \numflux{u_{\jj-1}}{u_{\jj}}(x_{\jj-\frac{1}{2}}) \jump{u^h}_{\jj-\frac{1}{2}} 
- \jump{g(u)}_{\jj-\frac{1}{2}}
\right) \\
&+ \left( F_{\text{cut}}(u) - F_{k-\frac{1}{2}} (u)
+ \numflux{u_{k-1}}{u_{\Kone}}(x_{k-\frac{1}{2}}) \jump{u^h}_{k-\frac{1}{2}} 
- \jump{g(u)}_{k-\frac{1}{2}}
\right) \\
& + \left( F_{k+\frac{1}{2}}(u) - F_{\text{cut}}(u) 
+ \numflux{u_{\Kone}}{u_{\Ktwo}}(x_{\text{cut}}) \jump{u^h}_{\text{cut}} 
- \jump{g(u)}_{\text{cut}}
\right).
\end{align*}
Due to the fluxes $F$ building a telescope sum and the usage of periodic boundary conditions, this implies
\begin{equation*}
a_h (u^h,u^h) = \Tone + \Ttwo
\end{equation*}
with
\begin{align*}
\Tone =&  \sum_{\jj \in \Iequi}
\left( \numflux{u_{\jj-1}}{u_{\jj}}(x_{\jj-\frac{1}{2}}) \jump{u^h}_{\jj-\frac{1}{2}} 
- \jump{g(u)}_{\jj-\frac{1}{2}}
\right), \\
\Ttwo =& \:
\numflux{u_{k-1}}{u_{\Kone}}(x_{k-\frac{1}{2}}) \jump{u^h}_{k-\frac{1}{2}} 
- \jump{g(u)}_{k-\frac{1}{2}} 
+ \numflux{u_{\Kone}}{u_{\Ktwo}}(x_{\text{cut}}) \jump{u^h}_{\text{cut}} 
 - \jump{g(u)}_{\text{cut}}. 
\end{align*}
Note that due to \eqref{eq: relation H and f} $\Tone, \Ttwo \ge 0.$

\textbf{Contribution of stabilization:} Now we consider the stabilization. We will not show $J_h(u^h,u^h) \ge 0$ but instead
$a_h(u^h,u^h)+J_h(u^h,u^h) \ge 0$. For the edge stabilization, we get
\begin{align*}
    \frac{1}{\etaKone} J_h^{0,\Kone}(u^h,u^h) 
    & = \left[\numflux{u_{k-1}}{u_{\Ktwo}} (x_{k-\frac{1}{2}})- \numflux{u_{k-1}}{u_{\Kone}}(x_{k-\frac{1}{2}})\right] \jump{u^h}_{k-\frac{1}{2}}\\
&\quad + \left[\numflux{u_{k-1}}{u_{\Ktwo}}(x_\Kcut)-\numflux{u_{\Kone}}{u_{\Ktwo}}(x_\Kcut)\right]
\jump{u^h}_{\text{cut}} \\
& = -\Ttwo + \Tthree 
\end{align*}
with
\begin{equation*}
\Tthree = \numflux{u_{k-1}}{u_{\Ktwo}} (x_{k-\frac{1}{2}}) \jump{u^h}_{k-\frac{1}{2}} 
- \jump{g(u)}_{k-\frac{1}{2}} 
+ \numflux{u_{k-1}}{u_{\Ktwo}}(x_\Kcut)  
 \jump{u^h}_{\text{cut}}   - \jump{g(u)}_{\text{cut}}. 
\end{equation*}
Since $\etaKone \in (0,1)$, we can later take care of the negative term
$-\eta_{\Kone} \Ttwo$ by adding the bilinear form $a_h$ to get
\begin{equation*}
a_h(u^h,u^h)-\etaKone \Ttwo = \Tone + (1-\etaKone) \Ttwo \ge 0.
\end{equation*}
It remains to examine $\Tthree$ and the volume stabilization term $J_h^{1,\Kone}$. 
Here, we make use of the assumption of the flux $\mathcal{H}$ being differentiable a.e. to write
\begin{equation*}
\frac{\dd}{\dd x}\numflux{u_{k-1}}{u_{\Ktwo}} = \numfluxa{u_{k-1}}{u_{\Ktwo}}\partial_x u_{k-1}+\numfluxb{u_{k-1}}{u_{\Ktwo}}\partial_x u_{\Ktwo}.
\end{equation*}
This implies
\begin{align*}
   \frac{1}{\etaKone} &J^{1,\Kone}_h(u^h,u^h) =\sum_{j\in \Ineigh} \matrixK(j)\int_{\Kone}\left(\numflux{u_{k-1}}{u_{\Ktwo}}-f(u_j)\right)\partial_x u_j\dd x\\
   &+\sum_{j\in \Ineigh} \matrixK(j)\int_{\Kone}\numfluxa{u_{k-1}}{u_{\Ktwo}}\,u_j \:\partial_x u_{k-1}\dd x
   +\sum_{j\in \Ineigh} \matrixK(j)\int_{\Kone}\numfluxb{u_{k-1}}{u_{\Ktwo}}\,u_j \:\partial_x u_{\Ktwo}\dd x\\
   &= \sum_{j\in \Ineigh} \matrixK(j)\int_{\Kone}\numflux{u_{k-1}}{u_{\Ktwo}}\partial_x u_j\dd x
    -\sum_{j\in \Ineigh} \matrixK(j)\left(g(u_j(x_{\Kcut}))-g(u_j(x_{k-\frac{1}{2}}))\right)\\
    &+ \sum_{j\in \Ineigh}\matrixK(j) \int_{\Kone}\left(\frac{\dd }{\dd x}\numflux{u_{k-1}}{u_{\Ktwo}}\right)u_j \: \dd x.
\end{align*}
Using $\frac{d}{dx}\left( \numflux{u_{k-1}}{u_{\Ktwo}}u_j \right) = \numflux{u_{k-1}}{u_{\Ktwo}}\partial_x u_j 
+ \frac{\dd }{\dd x}\numflux{u_{k-1}}{u_{\Ktwo}}u_j$, we get
\begin{align*}
  \frac{1}{\etaKone}  J^{1,\Kone}_h(u^h,u^h) = \sum_{j\in \Ineigh} \matrixK(j) & \bigl[\left(\numflux{u_{k-1}}{u_{\Ktwo}}u_j\right)(x_{\Kcut})
    -\left(\numflux{u_{k-1}}{u_{\Ktwo}}u_j\right)(x_{k-\frac{1}{2}}) .\\
    &-  g(u_j(x_{\Kcut}))+g(u_j(x_{k-\frac{1}{2}}))\bigr].
\end{align*}
Recall that
\begin{equation*}
\matrixK(k-1) = \matrixL_{\Kone}, \quad \matrixK(\Kone) = -1, \quad \matrixK(\Ktwo) = \matrixR_{\Kone} 
\end{equation*}
with $\matrixL_{\Kone},\matrixR_{\Kone} \in [0,1]$ and $\matrixL_{\Kone} + \matrixR_{\Kone} = 1$.
Then, skipping some tedious computations for brevity, we get
\begin{align*}
\frac{1}{\etaKone} J^{1,\Kone}_h(u^h,u^h) + \Tthree  =& \numflux{u_{k-1}}{u_{\Ktwo}}(x_{k-\frac{1}{2}})u_{k-1}(x_{k-\frac{1}{2}}) - \numflux{u_{k-1}}{u_{\Ktwo}}(x_{\text{cut}})u_{\Ktwo}(x_{\text{cut}}) \\
& \qquad - g(u_{k-1}(x_{k-\frac{1}{2}})) + g(u_{\Ktwo}(x_{\text{cut}})) \\
& + \matrixL_{\Kone} \left[ \numflux{u_{k-1}}{u_{\Ktwo}}(x_{\text{cut}})u_{k-1}(x_{\text{cut}}) 
- \numflux{u_{k-1}}{u_{\Ktwo}}(x_{k-\frac{1}{2}})u_{k-1}(x_{k-\frac{1}{2}}) \right.\\
&\left. \qquad  - g(u_{k-1}(x_{\text{cut}})) + g(u_{k-1}(x_{k-\frac{1}{2}}))\right] \\
&+ \matrixR_{\Kone} \left[ \numflux{u_{k-1}}{u_{\Ktwo}}(x_{\text{cut}})u_{\Ktwo}(x_{\text{cut}}) 
- \numflux{u_{k-1}}{u_{\Ktwo}}(x_{k-\frac{1}{2}})u_{\Ktwo}(x_{k-\frac{1}{2}}) \right.\\
&  \left. \qquad  - g(u_{\Ktwo}(x_{\text{cut}})) + g(u_{\Ktwo}(x_{k-\frac{1}{2}}))\right]\\
=& \Tfour + \Tfive
\end{align*}
with
\begin{align*}
\Tfour & =\matrixL_{\Kone}\left[\numflux{u_{k-1}}{u_{\Ktwo}}(x_{\Kcut})\left(u_{k-1}(x_{\Kcut})-u_{\Ktwo}(x_{\Kcut})\right) 
 -g(u_{k-1}(x_{\Kcut}))+ g(u_{\Ktwo}(x_{\Kcut}))\right]\\
\Tfive &=\matrixR_{\Kone}\left[\numflux{u_{k-1}}{u_{\Ktwo}}(x_{k-\frac{1}{2}})\left(u_{k-1}(x_{k-\frac{1}{2}})-u_{\Ktwo}(x_{k-\frac{1}{2}})\right)  -g(u_{k-1}(x_{k-\frac{1}{2}}))+g(u_{\Ktwo}(x_{k-\frac{1}{2}}))\right].
\end{align*}
Note that we use $\matrixL_{\Kone} + \matrixR_{\Kone} = 1$ here.
Again, $\Tfour, \Tfive \ge 0$ due to \eqref{eq: relation H and f}. 
In total, we get for the stabilization
\begin{equation*}
J^{0,\Kone}_h(u^h,u^h) + J^{1,\Kone}_h(u^h,u^h) =  -\etaKone \Ttwo + \etaKone \Tfour + \etaKone \Tfive.
\end{equation*}
Together with the bilinear form $a_h$, this gives
\begin{equation*}
    a_h(u^h,u^h) + J_h(u^h,u^h) = \Tone + (1-\etaKone) \Ttwo + \etaKone \Tfour + \etaKone \Tfive.
\end{equation*}
 As $\Tone, \Ttwo, \Tfour, \Tfive \ge 0$ and all prefactors are non-negative due to
 $0 < \etaKone < 1$, this concludes the proof.
\end{proof}

\begin{remark}
Let us consider the special case of the linear advection equation with $\beta=1$ and upwind flux. Then, $g(u)=\frac 1 2 u^2$ and $a_h(u^h,u^h)$ reduces to
\begin{equation*}
a_h(u^h,u^h) = \underbrace{\sum_{\jj \in \Iequi} \frac 1 2 \jump{u^h}^2_{\jj-\frac{1}{2}}}_{=\Tone} 
+ \underbrace{\frac 1 2 \jump{u^h}^2_{k-\frac{1}{2}} + \frac 1 2 \jump{u^h}^2_{\text{cut}}}_{=\Ttwo} .
\end{equation*}
In the stabilization several terms drop out, compare \eqref{eq: stabilization linear advection},
and we get
\begin{equation*}
\frac{1}{\etaKone}
J_h(u^h,u^h) \ = \underbrace{- \frac 1 2 \jump{u^h}^2_{k-\frac{1}{2}} - \frac 1 2 \jump{u^h}^2_{\text{cut}}}_{=- \Ttwo} + \underbrace{\frac{1}{2}  \left[ u_{\Ktwo}(x_{\text{cut}}) - u_{k-1}(x_{\text{cut}}) \right]^2}_{= \Tfour}.
\end{equation*}
When considering the sum $a_h(u^h,u^h) + J_h(u^h,u^h)$, we observe that the stabilization has the effect of replacing a certain portion, identified by $\etaKone$, of the `standard' jumps $(\Ttwo)$ at both edges $x_{k-\frac 1 2}$ and $x_{\text{cut}}$ of the small cut cell $I_{\Kone}$ by an `extended' jump $(\Tfour)$, evaluated at $x_{\text{cut}}$.  
\end{remark}


\section{Numerical results}\label{sec: numerical results}

In this section we present numerical results for both scalar conservation laws and systems of conservation laws. 
We will show results for piecewise constant polynomials in space as well as for higher order polynomials to assess
accuracy and stability of the proposed scheme.

To test convergence properties we need smooth solutions, which is non-trivial for, e.g., the compressible Euler equations. 
We will use 
\textit{manufactured solutions} for this purpose: we define a smooth function $\bfu(x,t)$ that we would like to be the solution of our system. Then we insert 
$\bfu(x,t)$ in the corresponding equations of the system.
This typically results in a non-zero source term $\mathbf{g}$ on the right hand side. As a consequence, instead of solving
\eqref{eq: conservation law}, we now solve the system
\begin{equation}\label{eq: conservation law with sourceterm}
\bfu_t+ \bff(\bfu)_x = \mathbf{g}\quad \text{ in } \Omega \times (0,T).
\end{equation}
The semi-discrete problem is then given by: Find $\bfu^h \in  \Vhp$ such that
\begin{equation*}
  \scpL{d_t\bfu^h(t)}{\bfw^h}+a_h\left(\bfu^h(t), \bfw^h\right) +  J_h(\bfu^h(t),\bfw^h) = \Source_h\left(\mathbf{g},\bfw^h\right)\quad\forall \, \bfw^h\in \Vhp,
\end{equation*}
 with
  \begin{equation*}
   \Source_h(\mathbf{g},\bfw^h) = \sum_{\jj
   \in \Iall} \int_{\jj} \mathbf{g} \cdot \bfw^h\dd{x}.
  \end{equation*}
  We will discretize this semi-discrete problem with a time stepping scheme whose order is chosen to match the order of the space discretization: When using the polynomial degree $p$ in space, we will use an SSP RK scheme of order $p+1$ in time. In particular, for piecewise constant polynomials in space we will use explicit Euler in time.
  Our test cases are extensions of the model problem $\Th$: we will use many cut cell pairs instead of using only one pair.
  Unless otherwise specified, we choose $\Omega = (0,1)$ and split \textit{every} cell $I_k$ between $x=0.1$ and $x=0.9$ in cut cell pairs  $(I_{\Kone},I_{\Ktwo})$ of lengths $\alpha_k h$ and $(1-\alpha_k)h$, where $\alpha_k\in (0,\frac{1}{2}]$ may be different for different $k$. We consider two cases:
  \begin{itemize}
    \item Case 1 ('$\alpha = 10^{-\square}$'): The cut cell fraction $\alpha_k$ is the same for all cut cell pairs, i.e. $\alpha_k \equiv \alpha$.
    \item Case 2 ('rand $\alpha$'): The cut cell fraction $\alpha_k$ varies and is computed randomly as $\alpha_k = 10^{-2}X_k$ with $X_k$ being a uniformly distributed random number in $(0,1)$.
\end{itemize}
  
  We compute the time step length according to \eqref{eq: time step}
using $\nu=0.4$ in all our experiments.
For systems, we compute the $L^1$ and $L^{\infty}$ error as
\begin{equation*}
    \norm{\bfu(\cdot,T)}_1 = \sum_{l=1}^m \norm{\bfu_l(\cdot,T)}_{L^1(\Omega)}, \quad
    \norm{\bfu(\cdot,T)}_{\infty} = \max_{1\le l \le m} \norm{\bfu_l(\cdot,T)}_{L^{\infty}(\Omega)}.
\end{equation*}
For the tests involving Burgers' equation and the linear system, we use the exact Riemann solver. For the Euler equations, we use the approximate Roe Riemann solver \cite{Toro}.
We implement periodic boundary conditions by setting $\bfu_0(x_{\frac{1}{2}}) = \bfu_N(x_{N+\frac{1}{2}})$ and 
$\bfu_{N+1}(x_{N+\frac{1}{2}}) = \bfu_1(x_{\frac{1}{2}})$. For transmissive boundary conditions we use
$\bfu_0(x_{\frac{1}{2}}) = \bfu_1(x_{\frac{1}{2}})$ and 
$\bfu_{N+1}(x_{N+\frac{1}{2}}) = \bfu_N(x_{N+\frac{1}{2}})$.

\subsection{Burgers equation}

We start with two tests for Burgers equation. In both cases, we initialize the solution with a sine curve.
In the first test, we force the solution to stay smooth. In the second test, we allow the shock and rarefaction waves to develop.

\begin{figure}[ht]
    \centering
    \begin{subfigure}[h]{0.37\linewidth}
    \includegraphics[width=\linewidth]{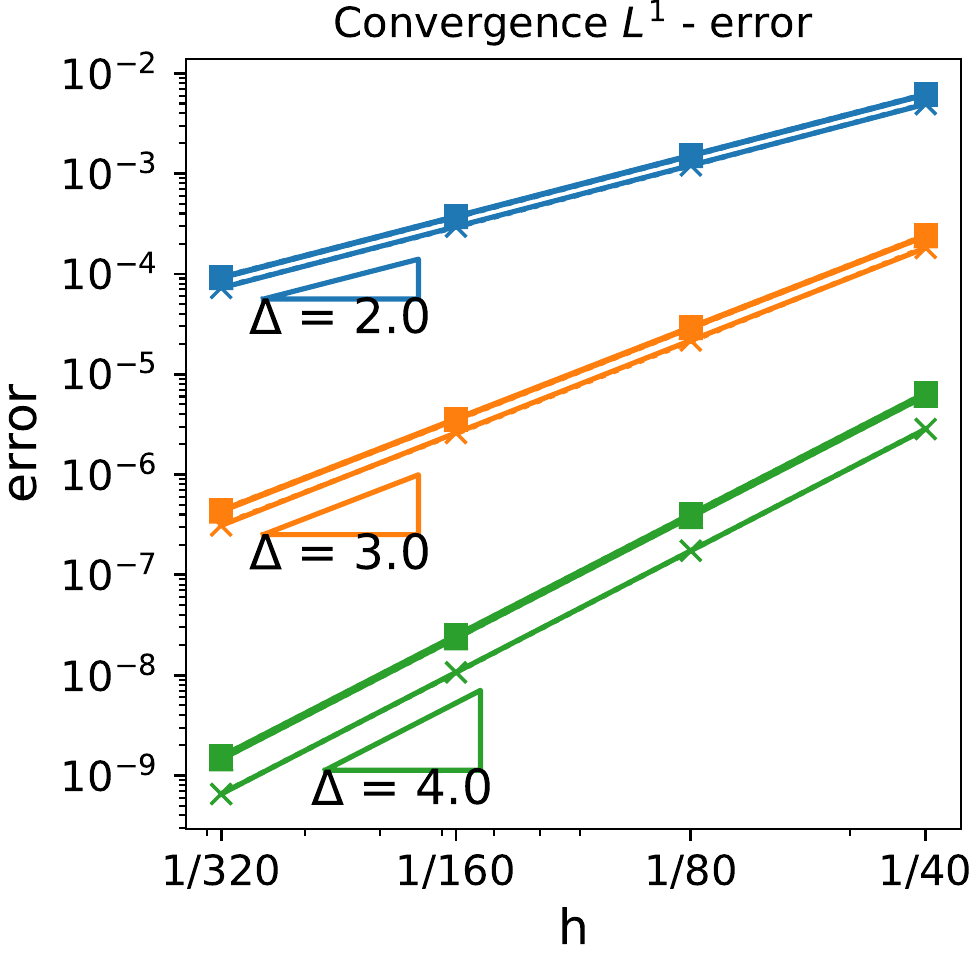}
    \end{subfigure}
    \begin{subfigure}[h]{0.18\linewidth}
    \includegraphics[width=\linewidth]{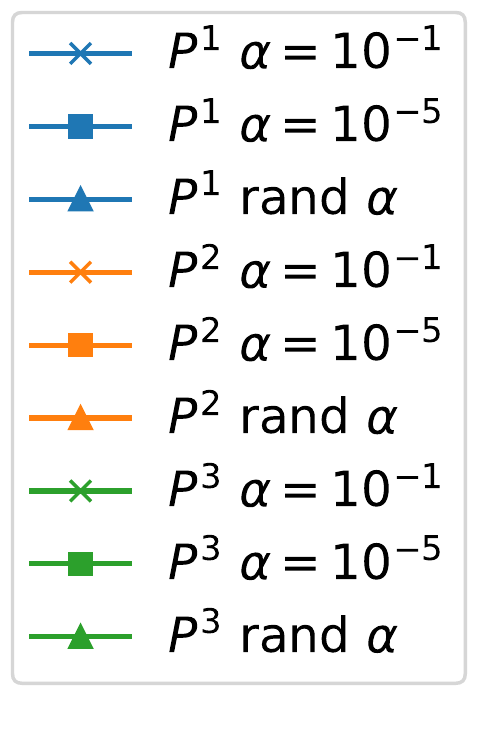}
    \end{subfigure}
    \begin{subfigure}[h]{0.37\linewidth}
    \includegraphics[width=\linewidth]{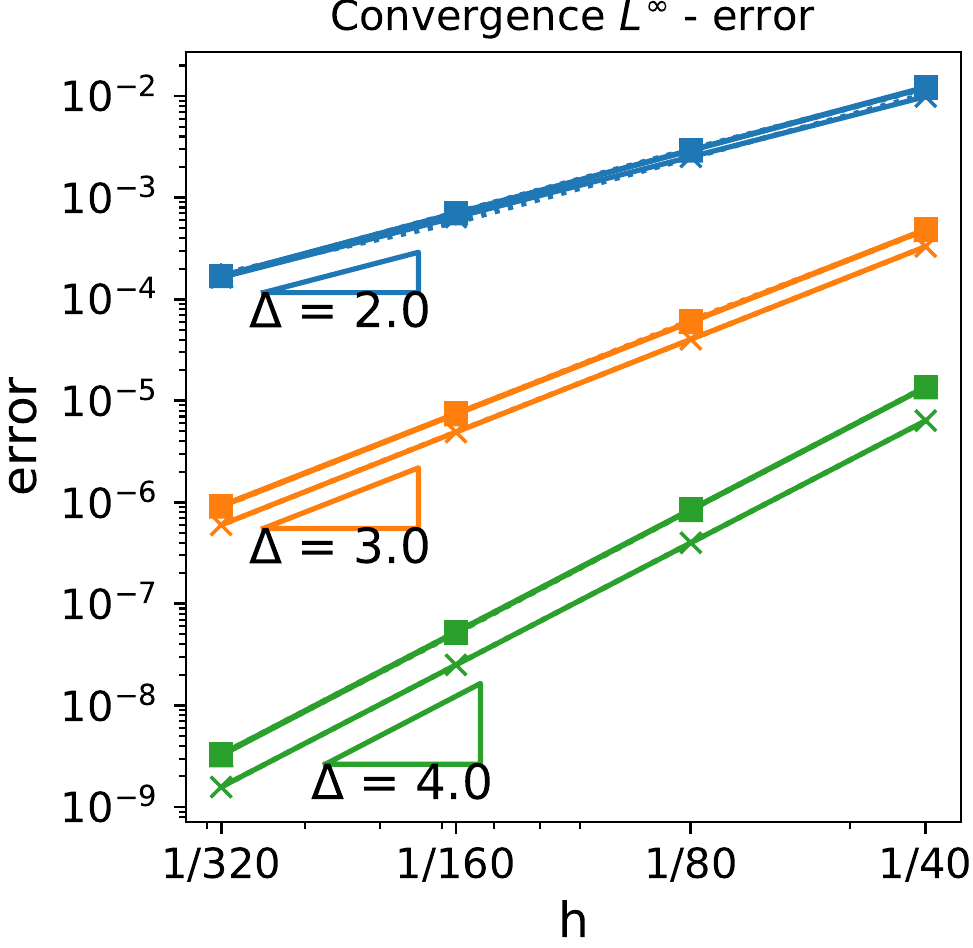}
    \end{subfigure}
    \caption{Convergence test for manufactured solution for Burgers equation: Error in the $L^1$ and $L^{\infty}$ norm.}\label{fig: err burgers}
\end{figure}

\subsubsection{Accuracy test with a manufactured solution}
We consider the manufactured solution 
\begin{equation*}
    u(x,t) = \sin(4\pi(x-t))
\end{equation*}
with periodic boundary conditions. 
This results in the source term
\begin{equation*}
    g(x,t) = 4\pi\cos(4\pi(x-t))\left(\sin(4\pi(x-t))-1\right).
\end{equation*}
In figure \ref{fig: err burgers} we show the error, measured in the $L^1$ and in the $L^\infty$ norm, for different values of the volume fractions $\alpha_k$ and different polynomial degrees at the final time $T=1$. We observe standard convergence rates, i.e., rates $p+1$ for polynomial degree
$p$ for both the $L^1$ and the $L^{\infty}$ norm. We also note that the error sizes for the different test cases
involving varying values of $\alpha_k$ are quite similar.

\subsubsection{Stability test}
Next, we consider a non-smooth problem. 
We choose the initial data
\begin{equation*}
    u_0(x) = \sin{(4\pi (x+0.5))}
\end{equation*}
with periodic boundary conditions and use $g=0$. As is well-known, these initial data result in the development of shock waves in the
regions where the derivative of $u_0$ is negative.

Figure \ref{fig: Burgers Sine} shows the solution at final time $T=0.1$ for different polynomial degrees for $\alpha_k$ being chosen randomly as specified above. The cut cell mesh was created from a mesh with
$N=100$ equidistant cells, and therefore contains 180 cells. For piecewise constant polynomials, the computed solution does not overshoot, consistent with the monotonicity result in theorem \ref{theorem: monotonicity}.
We also show the solution for $P^3$ polynomials, with and without the limiter. Without the limiter, the solution produces overshoot near the shock. Nevertheless, as is the case on a regular mesh, the numerical tests are stable and do not break despite using small cut cells. With limiter, the overshoot is gone.

\begin{figure}[ht]
    \centering
    \begin{subfigure}{0.4\linewidth}
    \includegraphics[width=\linewidth]{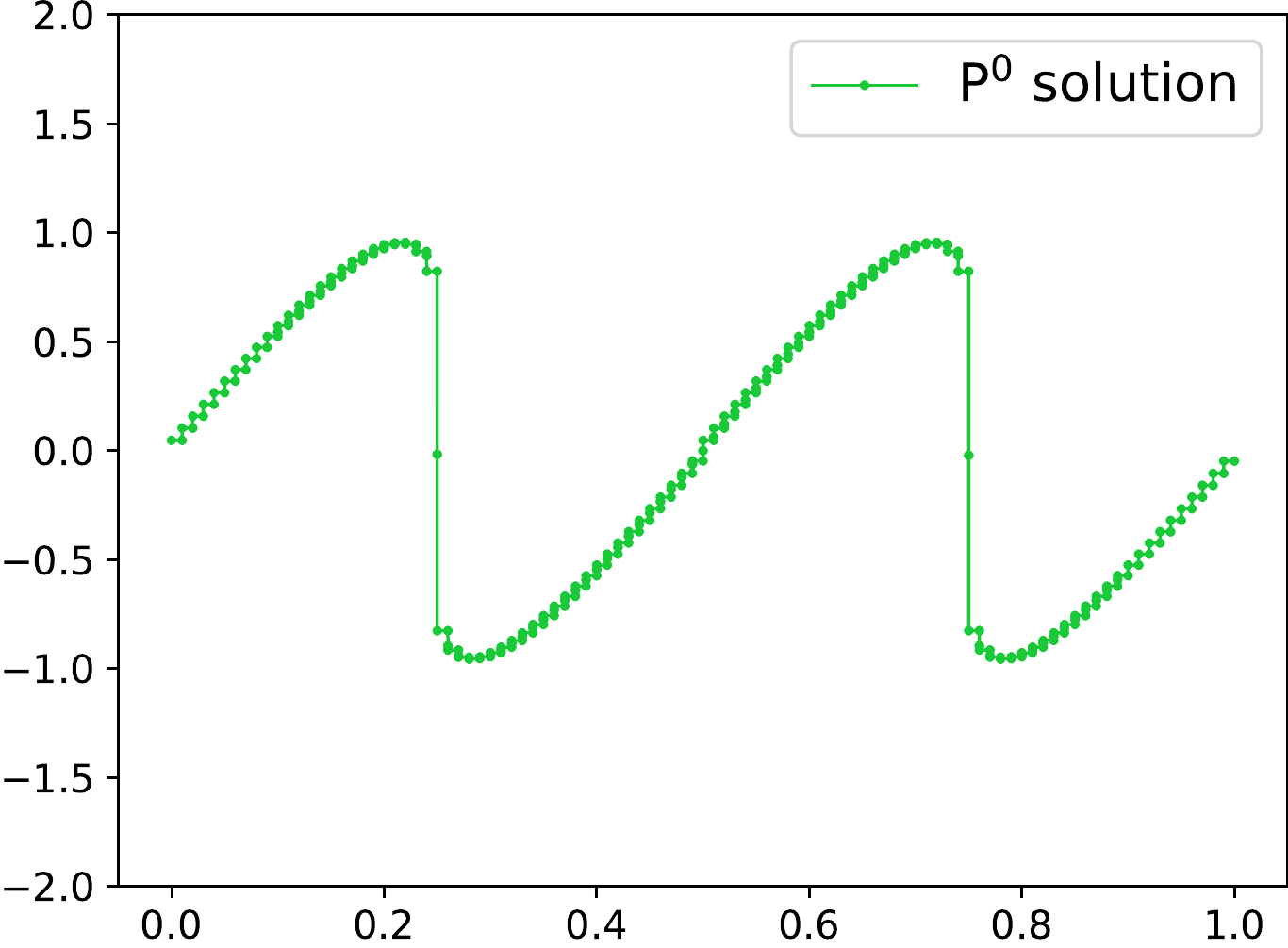}
    \end{subfigure}
    \begin{subfigure}{0.4\linewidth}
    \includegraphics[width=\linewidth]{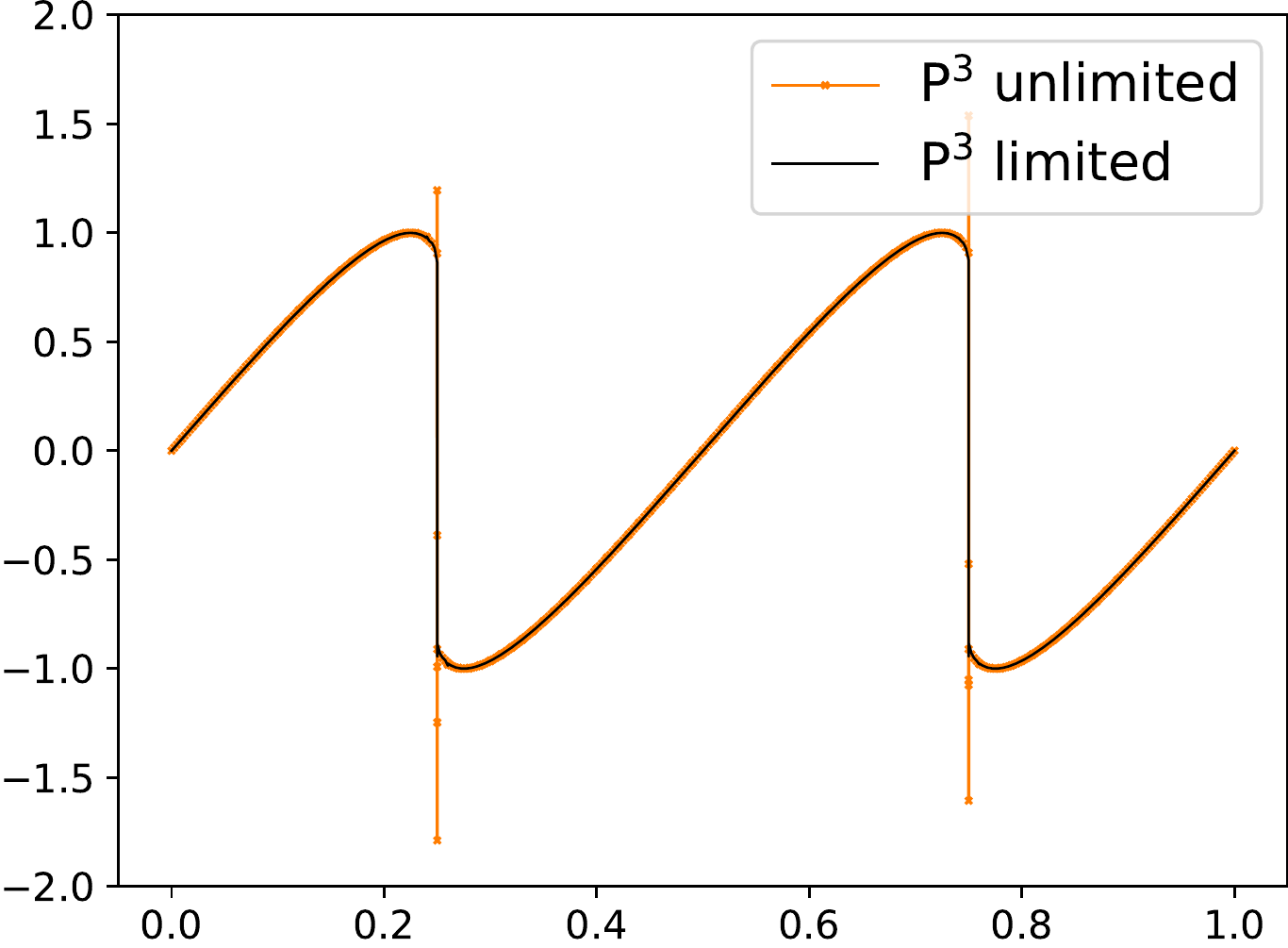}
    \end{subfigure}
    \caption{Stability test for Burgers equation: Solution at final time for piecewise constant polynomials (left) and piecewise cubic polynomials with and without a limiter (right).}\label{fig: Burgers Sine}
\end{figure}

\subsection{Linear systems}
We now consider the linear system given by equation \eqref{eq: lin system} with 
\begin{equation*}
\bfA=\begin{pmatrix*}[r]
4 & 2.5 & -7\\
-1 & 0.5 & 7\\
-0.5 & 1.25 & 1.5
\end{pmatrix*}
\quad \text{and} \quad
    \bfu_0(x,t) = \begin{pmatrix*}
        \sin(2\pi x) \\
        -\frac{1}{3}\cos(2\pi x)\\
        \frac{1}{2}\sin(2\pi x)
    \end{pmatrix*}.
\end{equation*}
The eigenvalues of $\bfA$ are $\lambda_1 = -2$, $\lambda_2=3$ and $\lambda_3 = 5$.
We again use periodic boundary conditions.

In figure \ref{fig: err lin system}, we show the errors in the $L^1$ and in the $L^{\infty}$ norm for piecewise linear, piecewise quadratic, and piecewise cubic polynomials
for different values of the volume fractions $\alpha_k$. 
As for Burgers equation, we observe convergence rates $p+1$ for polynomial degree $p$ for both the $L^1$ and $L^{\infty}$ error.
\begin{figure}[ht]
    \centering
    \begin{subfigure}[h]{0.37\linewidth}
    \includegraphics[width=\linewidth]{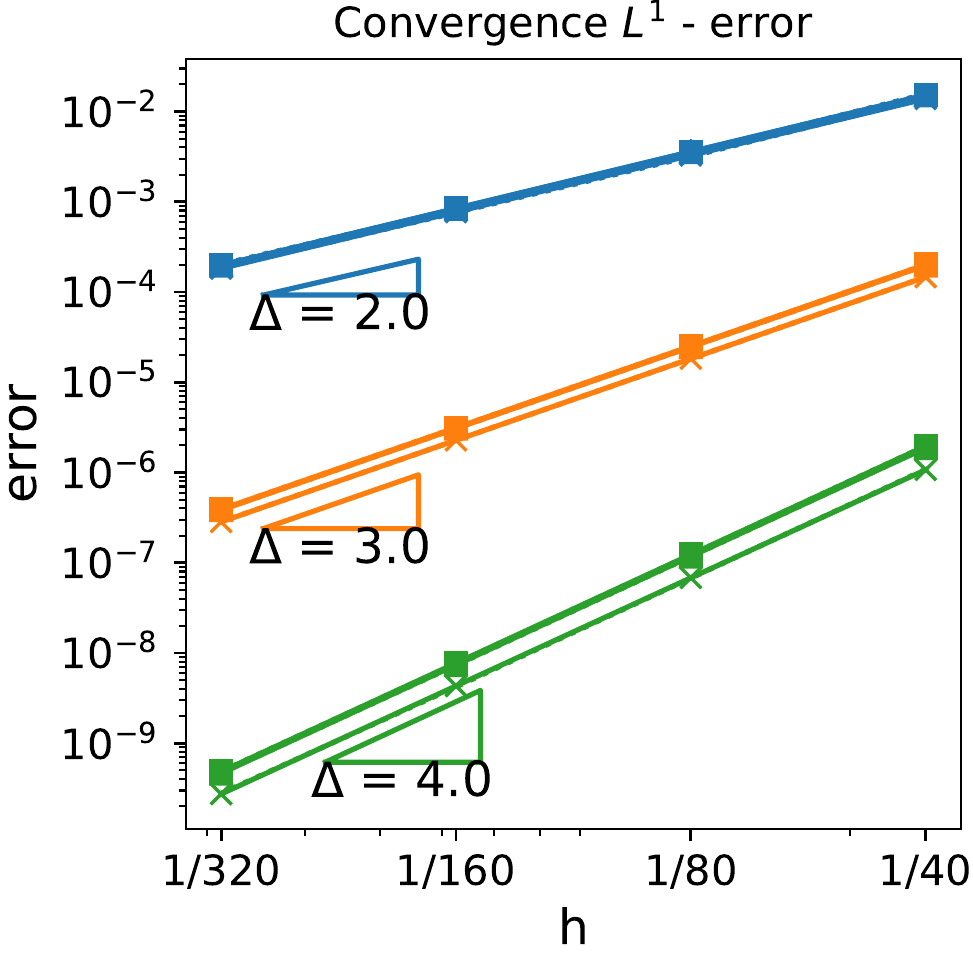}
    \end{subfigure}
    \begin{subfigure}[h]{0.18\linewidth}
    \includegraphics[width=\linewidth]{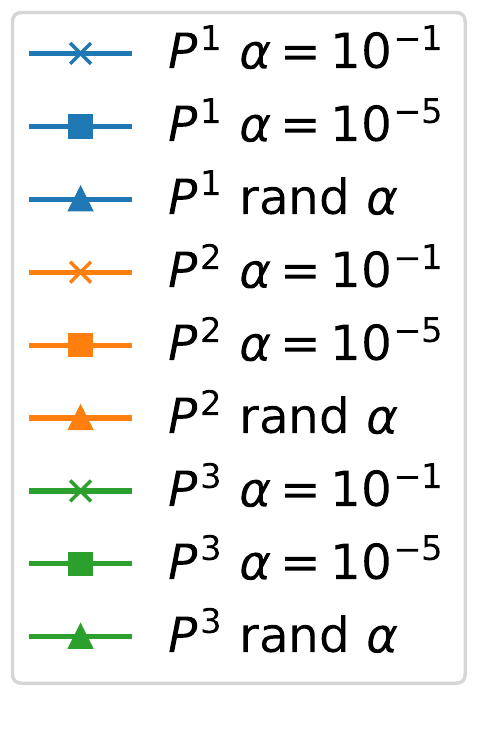}
    \end{subfigure}
    \begin{subfigure}[h]{0.37\linewidth}
    \includegraphics[width=\linewidth]{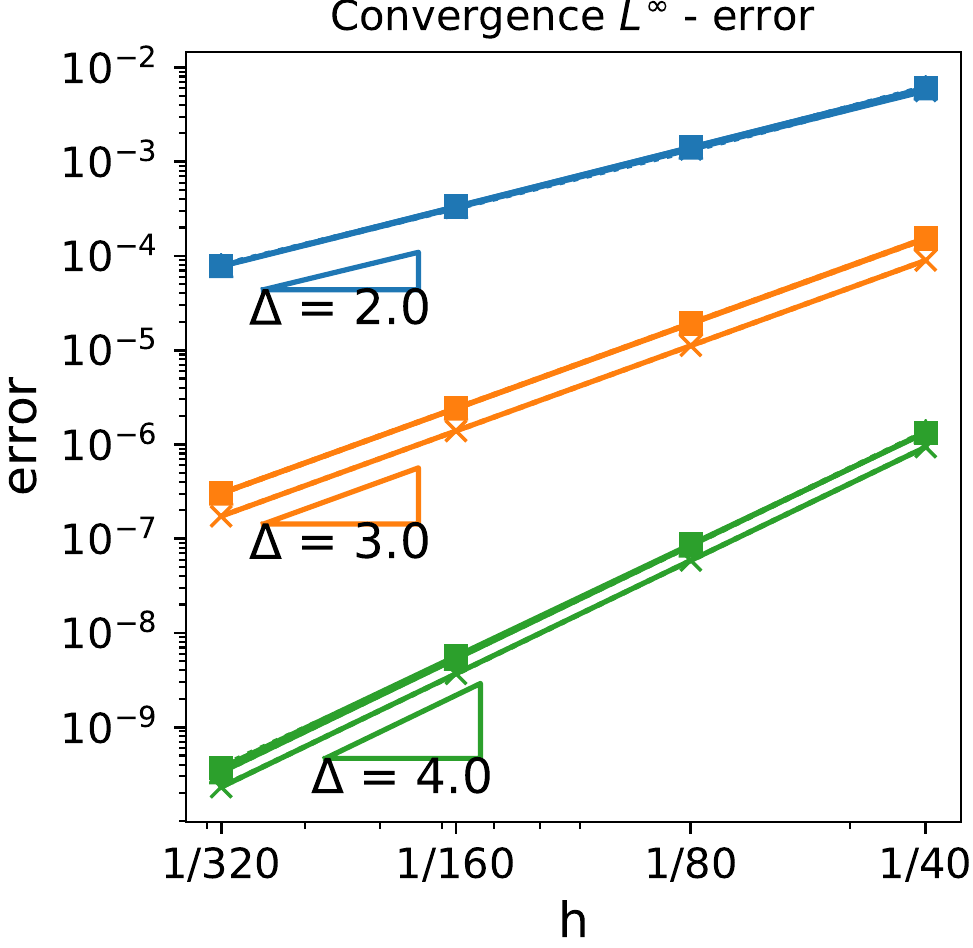}
    \end{subfigure}
    \caption{Convergence test for linear system: Error in the $L^1$ and $L^{\infty}$ norm.}
    \label{fig: err lin system}
\end{figure}
\subsection{Euler equations}
For the Euler equations we present two tests: a test with a smooth manufactured solution and the Sod shock tube test.
\subsubsection{Accuracy test with manufactured solution}
\begin{figure}
    \centering
    \begin{subfigure}[h]{0.37\linewidth}
    \includegraphics[width=\linewidth]{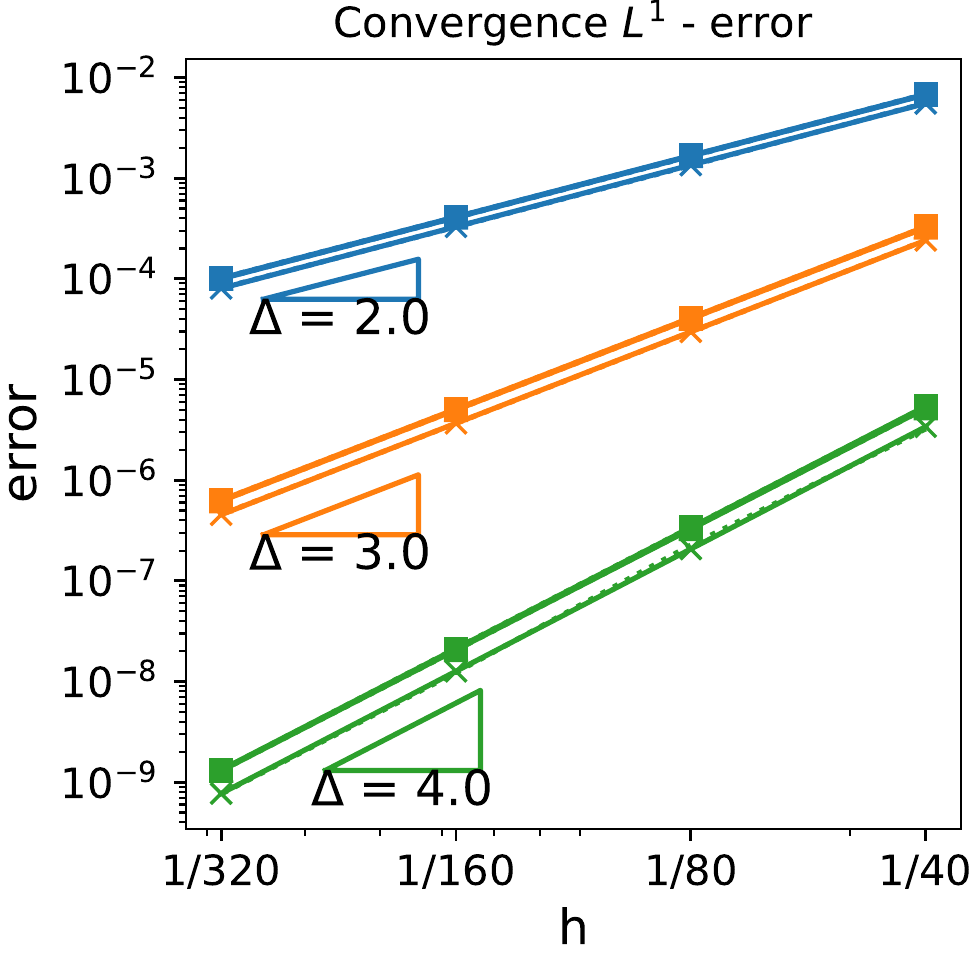}
    \end{subfigure}
    \begin{subfigure}[h]{0.18\linewidth}
    \includegraphics[width=\linewidth]{Kapitel_5/Lin_System/legend.pdf}
    \end{subfigure}
    \begin{subfigure}[h]{0.37\linewidth}
    \includegraphics[width=\linewidth]{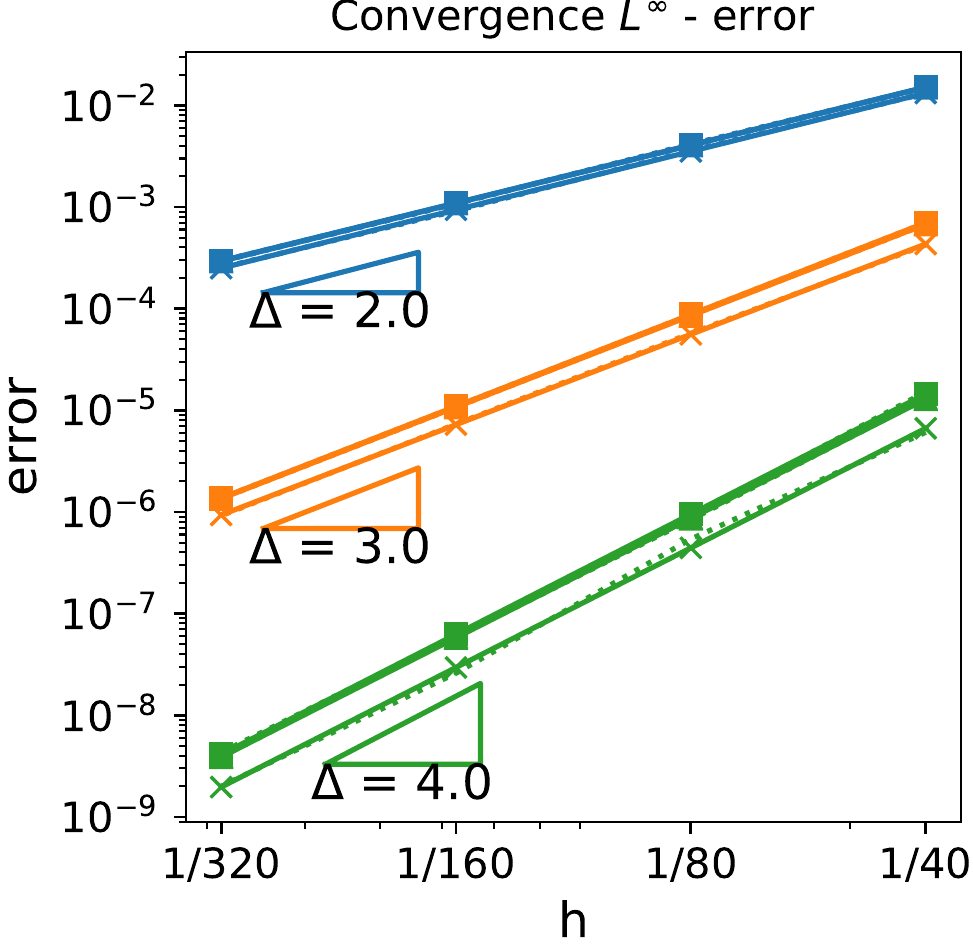}
    \end{subfigure}
    \caption{Convergence test for manufactured solutions for Euler equations: Error in the $L^1$ and $L^{\infty}$ norm.}
        \label{fig: err Euler system}
\end{figure}

We define the solution (in terms of primitive variables) as
\begin{equation*}
    \begin{pmatrix}
    \rho \\ v \\ p
    \end{pmatrix} = \begin{pmatrix}
    2+\sin (2\pi(x-t))\\
    \sin(2\pi(x-t))\\
    2+\cos(2\pi(x-t))
    \end{pmatrix}
\end{equation*}
together with periodic boundary conditions. The source term $\mathbf{g}(x,t)$ can be calculated by inserting the vector of conserved variables $\bfu(x,t)$ into equation \eqref{eq: conservation law with sourceterm} (but it is not given here due to its length). 

In figure \ref{fig: err Euler system} we show the $L^1$ and the $L^\infty$ error for different test cases at time $T=1$. Again, we see optimal convergence rates in the $L^1$ and in the $L^\infty$ norm for the different polynomial degrees.

\subsubsection{Sod shock tube test}
We conclude the numerical results with the well-known Sod shock tube test \cite{Toro}. 
The initial data are given by the following Riemann problem
\begin{equation*}
 \left(\rho, \rho v , E \right) = 
    \begin{cases}
    \left(1, 0, 2.5\right) &\text{ if } x<0,\\
    \left(0.125, 0, 0.25\right) &\text{ otherwise}.
    \end{cases}
\end{equation*}
For this test, we choose $\Omega = (-1,1)$ and use transmissive boundary conditions.
We discretize $\Omega$ with $N=100$ equidistant cells and split every cell in $[-0.75,0.75]$ into a pair of two cut cells with the volume fraction $\alpha_k$ chosen randomly as described above. We set $T=0.4$.

In figure \ref{fig: euler sod p0} we show the solution for density and for velocity at the final time using piecewise constant polynomials. 
As expected for $P^0$, the solution looks good but is quite diffusive.
Figure \ref{fig: euler sod p1} shows the solution for piecewise linear, limited polynomials.
We applied the limiter described in subsection \ref{sec: limiter} to the components of the conserved variables
and added a check to ensure that the pressure stays positive. Compared to the results for $P^0$, the results are significantly less diffusive while mostly being free of oscillations.

\begin{figure}[ht]
\centering
\includegraphics[width=0.65\linewidth]{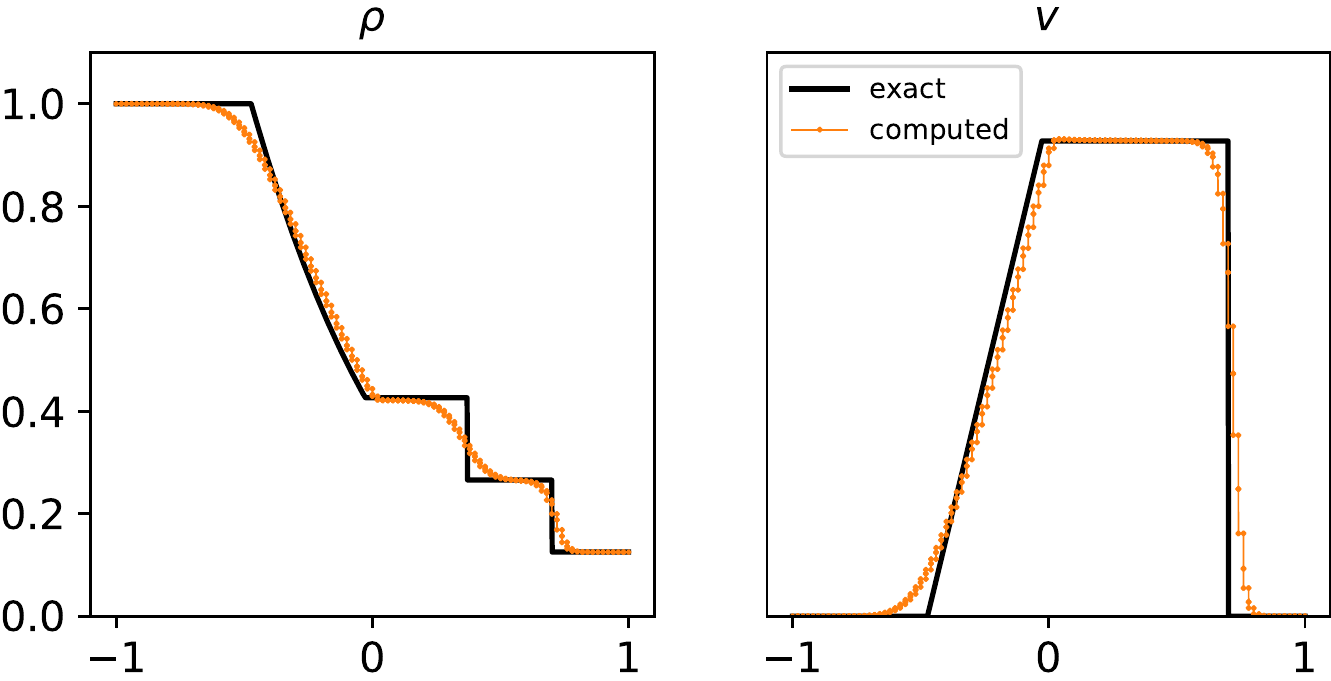}
\caption{Sod shock tube test: Numerical solution for density $\rho$ and velocity $v$ at final time using piecewise constant polynomials.}\label{fig: euler sod p0}
\end{figure}

\begin{figure}[ht]
\centering
\includegraphics[width=0.65\linewidth]{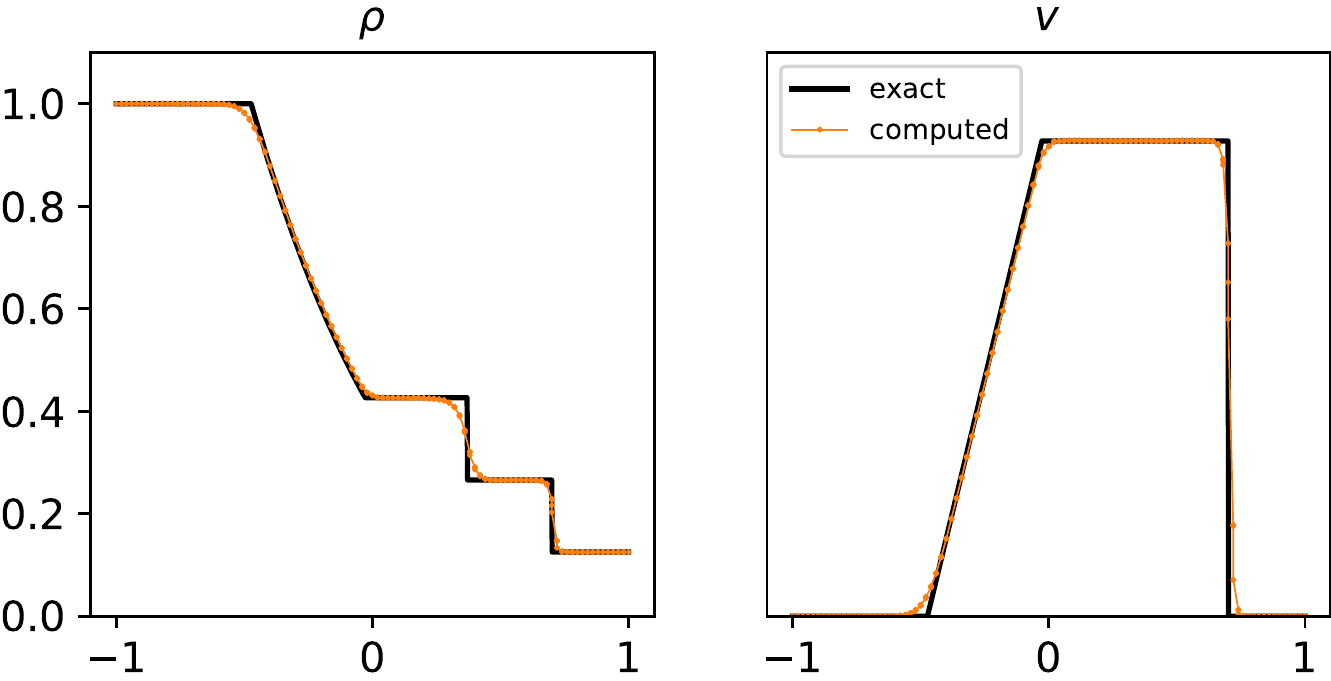}
\caption{Sod shock tube test: Numerical solution for density $\rho$ and velocity $v$ at final time using piecewise linear polynomials with limiter.}\label{fig: euler sod p1}
\end{figure}

\section{Conclusions and outlook}\label{sec: outlook}
In this contribution we have presented the extension of the DoD stabilization to non-linear problems and to higher order polynomials. To account for the latter, we have extended the support of test functions from small cut cells' neighbors into the small cut cells, compare $J_h^{1,\Kone}$ in \eqref{eq: def J1}. This stabilizes the derivatives on the cut cells' neighbors. To account for the changing flow directions in non-linear problems, we make use of Riemann solvers in both $J_h^{0,\Kone}$ and $J_h^{1,\Kone}$. Note also that both penalty terms treat the left and right neighbors of small cut cells in a symmetric way. 

For our new formulation we can show that the fully discrete, first-order scheme is monotone for scalar conservation laws. For the semi-discrete formulation, we have an $L^2$ stability result for arbitrary polynomial degree $p$. Our numerical results 
confirm that the DoD stabilized scheme has the same order of accuracy as standard RKDG schemes on equidistant meshes. Further, we observe robust behavior in the presence of shocks.

The choice of $J_h^{0,\Kone}$ followed in a fairly straightforward way from the choice of $J_h^{0,\Kone}$ for linear advection in \cite{DoD_SIAM_2020} by accounting for the changing flow directions. The design of $J_h^{1,\Kone}$ was significantly more complicated. The goal of ensuring $L^2$ stability has been a major guideline in the development of the terms. 

The next step will be the extension of the formulation to higher dimensions. Here, the main difficulty will consist in extending the penalty term $J_h^{1,\Kone}$ appropriately. Solving a Riemann problem in the interior of a cell in two dimensions is non-trivial. We believe that it will be necessary to replace this formulation by a suitable approximation, similarly to using approximate Riemann solvers instead of exact ones. We also believe that the results presented in this contribution are an essential step and a very good guideline towards reaching that goal.

\section*{Acknowledgments}
The authors would like to thank Christian Engwer, Andrew Giuliani, and Tim Mitchell for helpful discussions.
F.S. gratefully acknowledges support by the Deutsche Forschungsgesellschaft (DFG, German Research Foundation) - 439956613 (Hypercut).




\bibliographystyle{plain}
\bibliography{Literature}

\end{document}